\documentclass[a4paper,12pt]{scrartcl}
\usepackage[T1]{fontenc}
\usepackage[utf8]{inputenc}
\usepackage{mathtools}
\usepackage{enumitem}
\usepackage{amsthm,amssymb}
\usepackage[cal=boondoxo]{mathalfa}
\usepackage{appendix}
\usepackage{stmaryrd}
\usepackage{xspace}
\usepackage{mathabx}
\usepackage[a4paper]{geometry}
\usepackage{csquotes}
\usepackage{microtype}
\usepackage{cite}
\usepackage{hyperxmp}
\usepackage[unicode,hyperfootnotes=false,pdftex]{hyperref}

\hypersetup{
    pdftitle = {The non-linear sewing lemma III: Stability and generic properties},
    pdfauthor = {Antoine Brault; Antoine Lejay},
    pdfdate = {2020-04-20},
}

\newtheorem{proposition}{Proposition}
\newtheorem{theorem}{Theorem}

\newtheorem{lemma}{Lemma}
\newtheorem{corollary}{Corollary}
\theoremstyle{definition}
\newtheorem{definition}{Definition}
\newtheorem{notation}{Notation}

\newtheorem{hypothesis}{Hypothesis}
\newtheorem{ghypothesis}{Global Hypothesis}

\theoremstyle{remark}
\newtheorem{remark}{Remark}
\newtheorem{example}{Example}

\DeclarePairedDelimiterXPP{\normlip}[1]{}{\|}{\|}{_{\mathrm{Lip}}}{#1}
\DeclarePairedDelimiterXPP{\normsup}[1]{}{\|}{\|}{_{\infty}}{#1}
\DeclarePairedDelimiterXPP{\normL}[1]{}{\|}{\|}{_{\Lambda}}{#1}
\DeclarePairedDelimiterXPP{\normE}[1]{}{\|}{\|}{_{\cE}}{#1}
\DeclarePairedDelimiterXPP{\normhold}[2]{}{\|}{\|}{_{#1}}{#2}
\DeclarePairedDelimiterXPP{\normO}[1]{}{\|}{\|}{_{\mathcal{O}}}{#1}
\DeclarePairedDelimiterXPP{\generalnorm}[2]{}{\|}{\|}{_{#1}}{#2}
\DeclarePairedDelimiterXPP{\normp}[1]{}{\|}{\|}{_{p}}{#1}
\DeclarePairedDelimiterXPP{\normf}[2]{}{\|}{\|}{_{#1}}{#2}

\DeclarePairedDelimiter{\abs}{|}{|}
\DeclarePairedDelimiter\Paren()
\DeclarePairedDelimiter{\floor}{\lfloor}{\rfloor}

\newcommand{\fM}{\mathfrak{M}}
\newcommand{\fV}{\mathfrak{V}}
\newcommand{\id}{\mathfrak{i}}
\newcommand{\cF}{\mathcal{F}}
\newcommand{\cE}{\mathcal{E}}
\newcommand{\cM}{\mathcal{M}}
\newcommand{\cQ}{\mathcal{Q}}
\newcommand{\cR}{\mathcal{R}}
\newcommand{\cA}{\mathcal{A}}
\newcommand{\cP}{\mathcal{P}}
\newcommand{\cC}{\mathcal{C}}
\newcommand{\cG}{\mathcal{G}}
\newcommand{\cCb}{\mathcal{C}_{\mathrm{b}}}
\newcommand{\rTT}{\mathbb{T}_+}
\newcommand{\RR}{\mathbb{R}}
\newcommand{\NN}{\mathbb{N}}
\newcommand{\TT}{\mathbb{T}}
\newcommand{\PP}{\mathbb{P}}
\newcommand{\EE}{\mathbb{E}}
\newcommand{\uU}{\mathrm{U}}
\newcommand{\uV}{\mathrm{V}}
\newcommand{\uL}{\mathrm{L}}
\newcommand{\dd}{\mathrm{d}}
\newcommand{\bx}{\mathbf{x}}
\newcommand{\vd}{\,\mathrm{d}}

\newcommand{\dD}{\mathrm{D}}
\newcommand{\eqdef}{\mathbin{\vcentcolon=}}

\newcommand\given{\nonscript\:\delimsize\vert\nonscript\:\mathopen{}} 
\newcommand\SetSymbol[1][]{\nonscript\:#1\vert\nonscript\:\mathopen{}\allowbreak}
\DeclarePairedDelimiterX\Set[1]\{\}{%
  \renewcommand\given{\SetSymbol[\delimsize]}#1}

\DeclareMathOperator{\mesh}{Mesh}
\DeclareMathOperator{\Lip}{Lip}
\DeclareMathOperator{\osc}{osc}
\DeclareMathOperator{\Id}{Id}
\newcommand{\rpi}{\pi_+}

\newcommand{\cCloc}{\mathcal{C}_{\mathrm{loc}}}

\newcommand{\diff}{\mathfrak{d}}

\DeclarePairedDelimiterXPP{\normlipvarpi}[1]{}{\|}{\|}{_{\mathrm{Lip}\div\varpi}}{#1}
\DeclarePairedDelimiterXPP{\normsupvarpi}[1]{}{\|}{\|}{_{\infty\div\varpi}}{#1}

\newcommand{\cSA}{\mathcal{SA}}
\newcommand{\cN}{\mathcal{N}}
\newcommand{\cU}{\mathcal{U}}
\newcommand{\cO}{\mathcal{O}}
 \newcommand{\bB}{\mathbf{B}}

\begin{document}

\title{The non-linear sewing lemma III: Stability and generic properties}
\author{Antoine Brault\thanks{Université Paris Descartes, MAP5 (CNRS UMR 8145), 45 rue desSaints-Pères, 75270 Paris cedex 06, France,} \thanks{Center for Mathematical Modeling (CNRS UMI 2807), University of Chile,      
\texttt{abrault@dim.uchile.cl}} \and Antoine Lejay\thanks{Université de Lorraine, CNRS, Inria, IECL, F-54000 Nancy, France, \texttt{antoine.lejay@univ-lorraine.fr}}}

\maketitle

\begin{abstract}
    Solutions of Rough Differential Equations (RDE) may be defined as paths
    whose increments are close to an approximation of the associated flow. They
    are constructed through a discrete scheme using a non-linear sewing lemma.
    In this article, we show that such solutions also solve a fixed point
    problem by exhibiting a suitable functional. Convergence then follows from
    consistency and stability, two notions that are adapted to our framework.
    In addition, we show that uniqueness and convergence of discrete
    approximations is a generic property, meaning that it holds excepted for a
    set of vector fields and starting points which is of Baire first category.
    At last, we show that Brownian flows are almost surely unique solutions to
    RDE associated to Lipschitz flows. The later property yields almost
    sure convergence of Milstein schemes.
\end{abstract}


\section{Introduction}

Rough Differential Equations (RDE) are natural extensions of Ordinary Differential Equations (ODE)
to equations driven by rough signals~\cite{lyons98a,lyons02b,friz,friz14a}. 
More precisely, RDE are equations of type 
\begin{equation}
    \label{eq:rde:intro}
    y_{t,s}=a+\int_s^t f(y_{r,s})\vd \bx_{r},\ t\in[s,T],
\end{equation}
where $\bx$ is a $p$-\emph{rough path} lying above a continuous path $x$ of finite $p$-variation
living in a Banach space~$\uU$. The order $\floor{p}$ determines the tensor
space in which $\bx$ lives in and the iterated integrals of $x$ to use.  The
minimal regularity of the vector field $f$ also depends on $p$.  The
solution~$y$ is itself of finite $p$-variation living in a finite or infinite
Banach space~$\uV$.  One of the main feature of the theory of rough paths is
the continuity of the \emph{Itô map} $\bx\mapsto y$.  When $x$ is
differentiable, \eqref{eq:rde:intro} is understood as the ODE $y_t=a+\int_0^t
f(y_s)\dot{x}_s\vd s$. As for ODE, we recover Cauchy-Peano and Cauchy-Lipschitz
(or Picard -Lindelöf) type results, where existence follows from Schauder fixed
point theorem or from Picard fixed point theorem under stronger regularity
conditions on the vector field $f$. The later case implies uniqueness of
solutions as well as extra properties.

Existence of solutions to \eqref{eq:rde:intro} were first proved by T.~Lyons using a fixed point theorem~\cite{lyons98a}. 
In \cite{davie05a}, A.M. Davie proposed an alternative approach based on discrete approximations so that solutions
are constructed as limit of numerical schemes based on Taylor developments. P.~Friz and N.~Victoir \cite{friz2008,friz14a}
have proposed another approximation based on sub-Riemannian geodesics, yielding again the convergence
of numerical schemes. More recently, I.~Bailleul have developed a framework in which the central tools 
are flows associated to \eqref{eq:rde:intro} and their approximations \cite{bailleul12a,bailleul13b,bailleul17a}. 
By flows, we mean the family of solutions $a\in\uV\mapsto y_{t,s}(a)$ when the later satisfies 
$y_{t,s}\circ y_{s,r}=y_{t,r}$ for any $r\leq s\leq t$. 
The approximation of the flow proposed by I.~Bailleul, A.~M. Davie and P.~Friz-N.~Victoir are all different, 
although giving rise to the same flow.

In \cite{brault1,brault2}, we have proposed an \textquote{agnostic} framework for dealing directly
with flows without referring to a particular approximation. Only a broad condition  is
given on the approximations of the flows, called \emph{almost flows}, to obtain 
a \emph{non-linear sewing lemma}, a natural extension of the additive and multiplicative sewing lemmas~\cite{lyons98a,feyel}.
When the underlying space $\uV$ is finite dimensional,
a measurable flow may exist even when several solutions to \eqref{eq:rde:intro}
are known to exist \cite{brault1}. When the flow is Lipschitz, it is uniquely associated to any almost flow 
in the same quotient class called a \emph{galaxy}, a notion which reflects the 
\textquote{closeness} between the two objects. 
In \cite{brault2}, we have studied the properties of \emph{stable almost flows},
a condition ensuring that compositions of the almost flows over small times remains Lipschitz, uniformly
in the choice of the composition. The limiting flow is thus Lipschitz. 
We have also studied the relationship between stable almost flows
and solutions to~\eqref{eq:rde:intro}, which are unique in this case.

The goal of this article is threefold: 
\begin{itemize}[leftmargin=1em]
    \item We extend the notion of almost flow. 
We also continue our study of \emph{D-solutions}, that are paths $z$ solutions to \eqref{eq:rde:intro} satisfying
\begin{equation}
    \label{eq:intro:1}
    \abs{z_t-\phi_{t,s}(z_s)}\leq C\abs{t-s}^\theta,\ \forall s\leq t\text{ with } \theta>1,
\end{equation}
for an almost flow $\phi$. This notion of solution was introduced by A.~M.~Davie in \cite{davie05a}.
Here, we focus on continuity and approximations of D-solutions when $\phi$ is a stable almost
flow. Besides, we construct a functional 
$\Phi$ such that any D-solution solves the fixed point problem $z=\Phi(z)$. From this, 
we develop in our context the classical notions of \emph{consistency} and \emph{stability} \cite{lax,chartres}
which we relate to convergence.

More precisely, we construct a functional $\Phi$ such that any D-solution solves
the fixed point problem $z=\Phi(z)$. At the difference with the classical setting
for fixed point, $\Phi$ is defined \emph{only} on D-solutions. 

For a partition $\pi$, we also define a functional $\Phi^\pi$ such that
any solutions to $z^\pi=\Phi^\pi(z^\pi)$ are discrete D-solutions,
that is $z^\pi$ solves \eqref{eq:intro:1}
for times $s,t$ in the partitions. Such discrete D-solutions are constructed
explicitly through the numerical scheme $z^\pi_{t_{k+1}}=\phi_{t_{k+1},t_k}(z^\pi_{t_k})$
when $\pi=\Set{t_k}_{k=0}^n$.

By \emph{consistency}, we mean that any D-solution solves $z=\Phi^\pi(z)+\epsilon^\pi$
for a perturbative term $\epsilon^\pi$ that converges to $0$ when the mesh of the partition
converges to $0$. By \emph{stability}, we means that roughly $(\mathrm{Id}-\Phi^\pi)$ is invertible
with an inverse uniformly bounded with respect to $\pi$.   
Seen as a principle \cite{chartres}, 
the Lax equivalence theorem \cite{lax} is valid in many situations, including ours. It provides
a simple way to assert convergence through the study of consistency and stability.
We then show that the notion of \emph{stable almost flow}, introduced in \cite{brault2}, 
leads to the stability of $\Phi^\pi$.
The various estimates obtained in this part are the keys to fulfill our second objective.

\item We prove \emph{generic properties} associated to RDEs. When solved in an infinite
dimensional space, solutions to ODE are not necessarily unique \cite{dieudonne}, nor the Euler
scheme converges. Nevertheless, 
following some results due to W.~Orlicz \cite{orlicz} and developed later by several
authors, the set of vector fields and starting point points for which non-uniqueness/non-convergence
of the Euler scheme hold are of Baire first category. The key point is that discrete
approximations are uniformly approximated by discrete approximations in which vector fields is Lipschitz continuous.
We develop a similar approach for solutions to Young (when the driving path is of $p$-variation
with $p<2$) and rough (when the driving path is a rough path of finite $p$-variation with $2\leq p<3$).
Such results exploit properties developed in the first part of this article regarding stable almost flows.

\item We apply these results to Brownian flows to pursue the study of~\cite{davie05a}
by mixing them with considerations from H.~Kunita~\cite{kunita_saint_flour}.
In particular, we show that for any vector field, the solution to the Itô SDE $X_t=a+\int_0^t \sigma(X_s)\vd B_s$
for $\sigma\in\cCb^{1+\gamma}$, $\gamma>0$, is also the unique D-solution to the corresponding RDE and is then associated
to a Lipschitz flow. The notable points are that $\sigma$ is assumed 
to be less regular than for proving uniqueness through a Banach fixed point theorem;
and that properties of stable almost flows are not used here.
Besides, A.M. Davie proved that for almost 
any choice of a Brownian rough paths, with suitable conditions on the underlying space,
there exists a vector field for which several D-solutions exist. To summarize, 
there exist Lipschitz flows which are not related to stable almost flows. This question
was left open in \cite{brault2}.
\end{itemize}

\noindent\textbf{Outline. } In Section~\ref{sec:def}, we introduce objects and notations that we
use through all the article. In Section~\ref{sec:D-sol}, we define D-solutions, and show
that they are solutions to a fixed point problem involving suitable functionals whose consistency, 
stability and convergence is studied. Generic properties are studied in Section~\ref{sec:generic}. 
In Section~\ref{sec:br-flow}, we study Brownian flows and show that it is fitted for our frameworks.
We end with an appendix with general considerations on unbounded flows, boundedness of solutions
as well as uniqueness of D-solutions.


\section{Definitions and notations}
\label{sec:def}

We introduce some notations and global hypotheses (in force throughout the whole article)
which follows (partly) the ones of \cite{brault1,brault2}.

\begin{notation}[Simplex] 
    For $C$ an interval of $\RR$, we set $C^2_+\eqdef\Set{(s,t)\in C^2\given s\leq t}$
    and $C^3_+\eqdef\Set{(r,s,t)\in C^3\given r\leq s\leq t}$.
\end{notation}

\begin{notation} We use 
    \begin{itemize}[leftmargin=1em]
	\item Two non-decreasing functions $\delta$ and $\varpi$ from $\RR_+$ to $\RR_+$ with $\delta(0)=\varpi(0)=0$. 
	    We write indifferently $\delta_t$ or $\delta(t)$, $t\geq 0$, whenever it is convenient.
	\item A \emph{time horizon} $T>0$ and $\TT:=[0,T]$.
	\item A map $\omega:\rTT^2\to\RR_+$ (a \emph{control}) which is super-additive,
	    ($\omega_{r,s}+\omega_{s,t}\leq \omega_{r,t}$ for any $(r,s,t)\in\rTT^3$) and  
continuous close to its diagonal and such that $\omega_{s,s}=0$ for all $s\in\TT$.
    \end{itemize}
\end{notation}

\begin{ghypothesis}[Controls over growth and remainder]
    \label{hyp:4}
    For some $\varkappa\in(0,1)$,  $2\varpi(x/2)\leq \varkappa\varpi(x)$ for any $x\geq 0$.
\end{ghypothesis}

\begin{remark}
    \label{rem:3} 
    Since $\varkappa<1$, $\varpi(x)/x$ converges to $0$
    as $x$ converges to $0$.
\end{remark}

\begin{ghypothesis}[Time horizon]
    \label{hyp:time}
The time horizon $T$ satisfies 
\begin{equation}
    \label{eq:time}
    \varkappa+2\delta_T<1.
\end{equation}
\end{ghypothesis}

Let $\uV$ be a Banach space with the norm $\abs{\cdot}$ and $\id$ be the identity map from $\uV$ to~$\uV$.

\begin{notation}[Modulus of continuity, Lipschitz and Hölder norm] The \emph{modulus of continuity} 
    of a function $f:\uV\to\uV$ is 
    \begin{equation*}
	\osc(f,\delta)\eqdef \sup_{\substack{a,b\in\uV\\\abs{a-b}\leq \delta}}\abs{f(a)-f(b)}
	\text{ for any $\delta>0$.}
    \end{equation*}
    Its $\alpha$-Hölder  semi-norm ($0<\alpha\leq 1$) and its Lipschitz semi-norm
    are defined as 
    \begin{equation*}
	\normhold{\alpha}{f}\eqdef \sup_{a\neq b}\frac{\abs{f(a)-f(b)}}{\abs{a-b}^\alpha}
	\text{ and }
	\normlip{f}\eqdef \sup_{a\neq b}\frac{\abs{f(a)-f(b)}}{\abs{a-b}}
    \end{equation*}
    when these quantities are finite. Moreover, if $f$ is bounded, we denote
    $\normsup{f}:=\sup_{a}\abs{f(a)}$.
\end{notation}

We consider several families $\chi$ of objects indiced by $\rTT^2$ (almost flows, control, ...). 
    When these objects are functions from $\uV$ to $\uV$, 
    we write the pair $(r,t)\in\rTT^2$ in reverse order, that is $\chi_{t,r}$, 
    as the composition of functions is usually written from right to left. 
    Other objects are written with indices in order, that is $\chi_{r,t}$.

\begin{definition}[Functions of class $\cO$]
\label{not:6}
A function $\chi$ from $\rTT^2$ to $\cC(\uV,\uV)$ is said to be \emph{of class $\cO$} if 
there exists a constant $C\geq 0$ such that
\begin{equation}
    \label{eq:defO}
    \osc(\chi_{t,s},L\varpi(\omega_{r,s}))\leq C\delta_T (1+L)\varpi(\omega_{r,t}),
    \ \forall (r,s,t)\in\rTT^3,\ \forall L\geq 0.
    \end{equation}
The smallest constant $C$ such that \eqref{eq:defO} holds is denoted by $\normO{\chi}$.
\end{definition}
\begin{definition}[Semi-norm on functions of class $\cO$]
    We define
    \begin{equation*}
	\cO(\uV,\uV)\eqdef \Set{\chi:\rTT^2\to\cC(\uV,\uV)\given \chi\text{ is of class }\cO }, 
    \end{equation*}
    which is a vector space with a semi-norm $\normO{\cdot}$.
\end{definition}
\begin{example} 
    \label{ex:1}
    Let $\chi_{t,r}$ be Lipschitz with $\normlip{\chi_{t,r}}\leq K$ 
    for any $(r,t)\in\rTT^2$.
    Then $\chi\in\cO(\uV,\uV)$ with $\normO{\chi}=K/\delta_T$.
\end{example}
\begin{example}
    \label{ex:2}
    Let $x:\TT\to\uU$ be $\alpha$-Hölder continuous and $f:\uV\to L(\uU,\uV)$
    be $\gamma$-Hölder continuous with $\theta\eqdef \alpha(1+\gamma)>1$.
    Let $\varpi(x)\eqdef x^\theta$ and $\omega_{s,t}=t-s$.
    For $a\in\uV$ and $(s,t)\in\rTT^2$, 
    set
    $\chi_{t,s}(a)\eqdef f(a)x_{s,t}$, where $x_{s,t}\eqdef x_t-x_s$. Then 
    \begin{equation}
	\label{eq:51}
	\abs{\chi_{t,s}(a)-\chi_{t,s}(b)}
	\leq \normhold{\gamma}{f}\cdot\normhold{\alpha}{x}
	\abs{a-b}^\gamma(t-s)^\alpha,\ \forall (s,t)\in\rTT^2.
    \end{equation}
    With $\delta_T\eqdef \normhold{\gamma}{f}\cdot\normhold{\alpha}{x}T^{\alpha\gamma^2}$, 
    it follows from \eqref{eq:51} that $\chi$ is of class $\cO$ with 
    $\normO{\chi}\leq\gamma^\gamma/(1-\gamma)^{1-\gamma}$.
\end{example}
\begin{notation}
    \label{not:1}
    Let $\cF[\delta]$ be the class of families $\phi\eqdef\Set{\phi_{t,s}}_{(s,t)\in\TT_+^2}$
    of functions from~$\uV$ to~$\uV$ which satisfy
    \begin{gather}
	\label{eq:def:1}
	\phi_{t,s}=\id+\widehat{\phi}_{t,s}\text{ with }\widehat{\phi}\in\cO(\uV,\uV)
	\text{ and }\normO{\widehat{\phi}}\leq 1,\\
	\label{eq:def:3}
	\normsup{\widehat{\phi}_{t,s}}\leq \delta_{t-s}, \ \forall (s,t)\in\rTT^2.
    \end{gather}
   The set $\cF\eqdef\bigcup_{\delta}\cF[\delta]$, union over all the
    functions $\delta$ 
    as in Global Hypothesis~\ref{hyp:4} (which is stable under addition), is equipped with the distance
\begin{equation}
	\label{eq:def:5}
	d_\infty(\phi,\psi)\eqdef\sup_{(s,t)\in\rTT^2} \sup_{a\in\uV}   \abs{\phi_{t,s}(a)-\psi_{t,s}(a)}.
    \end{equation}
\end{notation}

\begin{remark}
    In Appendix~\ref{sec:aflinear}, we justify that assuming that $\widehat{\phi}$ is bounded
    unlike in~\cite{brault1} can be done without losing generality.
\end{remark}

\begin{definition}[Galaxy]
    \label{def:gal:1}
    Let $\phi,\psi\in\cF$. 
    We say that $\phi$ and $\psi$ are in the same \emph{galaxy}
    if there exists $K\geq 0$ such that 
    \begin{equation}
	\label{eq:gal:1}
	\normsup{\phi_{t,s}-\psi_{t,s}}\leq K\varpi(\omega_{s,t}),\ \forall(s,t)\in\rTT^2. 
    \end{equation}
\end{definition}

\begin{definition}[Almost flow]
    \label{def:almost-flow}
    We fix $M\geq 0$. 
    Let $\cA[\delta,M]$ be the set of $\phi\in\cF[\delta]$ such that 
    \begin{gather}
	\label{eq:def:2}
	\normsup{\diff\phi_{t,s,r}}\leq M\varpi(\omega_{r,t}),
	\ \forall (r,s,t)\in\TT^3_+\\
	\text{ with }
	\label{eq:phitsr}
	\diff\phi_{t,s,r}\eqdef \phi_{t,s}\circ\phi_{s,r}-\phi_{t,r}.
    \end{gather}
    We write $\cA\eqdef\bigcup_{\delta,M}\cA[\delta,M]$. 
    An element of $\cA$ is called an \emph{almost flow}.
\end{definition}

\begin{remark} 
    Combining Examples~\ref{ex:1} and \ref{ex:2}, it is easily
    seen that this definition generalizes the one of \cite{brault1}.
\end{remark}

\begin{definition}[Flow]
    A flow is a family $\psi:\rTT^2\times\uV\to\uV$ which satisfies
    $\diff\psi_{t,s,r}(a)=0$ for any $a\in\uV$ and any $(r,s,t)\in\rTT^3$.
\end{definition}

\begin{remark}
    For $i=1,2,3$, let us define $\fM_i$ be the maps from $\rTT^i\times \uV$ to $\uV$. 
    The operator $\diff$ transforms maps in $\fM_2$ to maps in $\fM_3$. 
    It is a non-linear generalization of the sewing operator introduced by M.~Gubinelli in \cite{gub04}. 
    We use it as a shorthand. Yet it has also the following meaning. 
    For a family of invertible maps $\alpha$ in $\fM_1$, we set 
    $\diff\alpha_{t,s}=\alpha_t\circ\alpha_s^{-1}$, $(s,t)\in\rTT^2$ 
    so that $\diff\alpha\in\fM_2$. Conversely, for an invertible 
    flow $\psi\in\fV$, we set $\alpha_t\in\psi_{t,0}$, $t\in\TT$ 
    so that $\diff\alpha=\psi$.
    Hence, invertible flows belong both to the range of $\diff:\fM_1\to\fM_2$ and
    the kernel of $\diff:\fM_3\to\fM_2$. When $\diff\phi$ is \textquote{close}
    to $0$ for an almost flow $\phi$, a \emph{non-linear sewing map} projects $\phi$
    to a flow $\psi$, which thus satisfies $\diff\psi=0$.
\end{remark}

\begin{notation} 
    \label{not:partition}
    The elements of a partition $\pi=\Set{t_i}_{i=0}^{n}$
    of $\TT$ are written either as the points $t_i$ or as the
    close intervals $[t_i,t_{i+1}]$ of successive points. 
    
    For a family $\Set{y_t}_{t\in\TT}$,
    we write $y_i\eqdef y_{t_i}$ when no ambiguity arises. We use the same convention for 
    functions over $\rTT^2$ or $\rTT^3$.

    For a family 
    $\Set{f_{s,t}}_{(s,t)\in\rTT^2}$, 
    we write either $\sum_{i=0}^{n-1} f_{i,i+1}$ or 
    $\sum_{[u,v]\in\pi} f_{u,v}$ instead of $\sum_{i=0}^{n-1} f_{t_i,t_{i+1}}$
    when there is no ambiguity.
\end{notation}
\begin{definition}[Solution in the sense of Davie, or D-solution]
    Let $n\geq 1$.
    For an almost flow $\phi\in\cA$, a partition $\pi=\Set{t_k}_{k=0}^n$ of $\TT$
    and $K\geq 0$, 
    we denote by $\cP_\pi[\phi,a,K]$ the set of $V$-valued families 
$\Set{y_{t_k}}_{k=0,\dotsc,n}$ such that $y_0=a$ and 
    \begin{equation}
	\label{eq:1disc}
    \abs{y_j-\phi_{j,i}(y_i)}\leq K \varpi(\omega_{i,j}),\ 
    \forall0\leq  i\leq j\leq n.
    \end{equation}
    We also set $\cP_\pi[\phi,a]\eqdef\bigcup_{K\geq 0}\cP_\pi[\phi,a,K]$.

    Similarly, we denote by $\cP[\phi,a]$ the set of paths $y\in\cC(\TT,\uV)$ with $y_0=a$ and 
    \begin{equation}
	\label{eq:1}
    \abs{y_t-\phi_{t,s}(y_s)}\leq K\varpi(\omega_{s,t}),\ \forall (s,t)\in\rTT^2,
    \end{equation}
    for some constant $K\geq 0$.

    The elements of $\cP_\pi[\phi,a]$ and $\cP[\phi,a]$ are 
    called \emph{solutions in the sense of Davie}, which we shorten by \emph{D-solutions}.
\end{definition}
\begin{definition}[Numerical scheme]
    \label{def:numerical-scheme}
    Given a partition $\pi=\Set{t_i}_{i=0}^n$ of $\TT$, 
    the \emph{numerical scheme} of an almost flow $\phi\in\cA$ is the
    sequence $\Set{y_{t_k}}_{k=0,\dotsc,n}$ constructed iteratively by 
    \begin{equation*}
	y_0=a\text{ and }y_{t_{k+1}}=\phi_{t_{k+1},t_{k}}(y_{t_k}),\ k=0,\dotsc,n-1.
    \end{equation*}
\end{definition}

We now define the notion of convergence of partitions.

\begin{definition}[Mesh and convergence]
    \label{def:mesh}
For a partition $\pi=\Set{t_i}_{i=0}^n$ of $\TT$, we define its \emph{mesh} 
by $\mesh\pi\eqdef\max{i=0,\dotsc,n-1}\Set{t_{i+1}-t_i}$. 
This define an order on partitions: $\sigma\leq \pi$ if $\mesh\sigma\leq \mesh\pi$.
A family $\Set{a_\pi}_{\pi}$ with values in a metric space $(\uV,d)$
is said to \emph{converge} to $a\in\uV$ whenever for any $\epsilon>0$
there exists a partition $\pi$ such that for any $\sigma\leq \pi$, 
$d(a_\sigma,a)\leq \epsilon$. 
\end{definition}

\begin{remark}
    Inclusion defines another partial order on partitions \cite{mcshane52a}. We do not use
    it, except as some tool in some proofs.
\end{remark}

\section{Stability results on D-solutions}

\label{sec:D-sol}

\subsection{Space of D-solutions}

We start by giving some precisions on the discrete and continuous spaces of D-solutions.

\begin{lemma}[{The spaces $\cP_{\pi}[\phi,a]$ are not empty}]
    \label{lem:2}
    For any almost flow $\phi\in\cA[\delta,M]$,  
    for any partition $\pi$ of $\TT$ and any $a\in\uV$, 
    the numerical scheme~$y^\pi$ associated to $\phi$ with $y^\pi_0=a$ 
    belongs to $\cP_\pi[\phi,a,L]$ with
    \begin{equation}
	\label{eq:L}
	L\eqdef \frac{2(\delta_T+M)}{1-\varkappa-2\delta_T}.
      \end{equation}
    Moreover, if $\psi$ is in the same galaxy as $\phi$, then $\cP_\pi[\psi,a]=\cP_\pi[\phi,a]$.
\end{lemma}

The proof of this result is a variant of the one of the Davie lemma given in \cite{brault1,brault2}.

\begin{proof}
    We set $U_{i,j}\eqdef \abs{y_j-\phi_{j,i}(y_i)}$ for $i\leq j$. Following \cite{davie05a,brault1}, 
    we proceed by induction on $j-i$. First, we remark that  $U_{i,i}=U_{i,i+1}=0$.
    Second, for 
    $i\leq j\leq k$ with $i<k$, 
    \begin{multline}
	\label{eq:39}
	y_k-\phi_{k,i}(y_i)
	\\
	=y_k-\phi_{k,j}(y_j)
	+\phi_{k,j}(y_j)-\phi_{k,j}(\phi_{j,i}(y_i))
	+\phi_{k,j}(\phi_{j,i}(y_i))-\phi_{k,i}(y_i)
	\\
	=
	y_k-\phi_{k,j}(y_j)
	+y_j-\phi_{j,i}(y_i)
	+\widehat{\phi}_{k,j}(y_j)-\widehat{\phi}_{k,j}(\phi_{j,i}(y_i))\\
	+\phi_{k,j}(\phi_{j,i}(y_i))-\phi_{k,i}(y_i).
    \end{multline}

    Our induction hypothesis is that  
    $U_{i,j}\leq L\varpi(\omega_{i,j})$ when $\abs{j-i}\leq m$ for some level $m$, where $L$ is defined in \eqref{eq:L}. This is true for $m=0,1$.

    Assume that the induction hypothesis is true whenever $j-i\leq m$
    for a level $m\geq 1$.
    We fix $i<k$ such that $\abs{k-i}\leq m+1$. We are going
      to show that $U_{i,k}\leq L\varpi(\omega_{i,k})$.

      If $\omega_{i,k}=0$, it follows by super-additivity of
      the control $\omega$ that $\omega_{i,k-1}=\omega_{k-1,k}=0$. This implies according to induction hypothesis that
      $U_{i,k-1}=U_{k-1,k}=0$. Then, using \eqref{eq:39} with
      $(i,j,k)=(i,k-1,k)$ and \eqref{eq:defO}, we get
      \begin{equation}
        \label{eq:5}
        U_{i,k}\leq \delta_T(1+L)\varpi(\omega_{i,k})+M\varpi(\omega_{i,k}).
      \end{equation}
      It follows that $U_{i,k}=0$, therefore $U_{i,k}\leq
      L\varpi(\omega_{i,k})$ holds.
      
      If $\omega_{i,k}>0$, let us define $j^*\eqdef \inf\left\{j\in
      \Set{i+1,\dotsc,k}\textrm{ such
        that } \omega_{i,j}>\frac{1}{2}\omega_{i,k}\right\}$.
        It follows of our definition of $j^*$ and from the
    super-additivity of $\omega$ that $\omega_{j^*,k}\leq
    \frac{1}{2}\omega_{i,k}$ and $\omega_{i,j^*-1
    }\leq \frac{1}{2}\omega_{i,k}$.
    
    We consider two cases : either $j^*<k$ or $j^*=k$.
    For the first case, using the fact that $\phi$ is an almost flow,
    \eqref{eq:defO} for $(r,s,t)=(i,j^*,k)$ and the equality~\eqref{eq:39} when
    $j=j^*$,
    \begin{align}
      \label{eq:2}
      U_{i,k}&\leq
      U_{i,j^*}+U_{j^*,k}+\osc\left(\widehat{\phi}_{k,j^*},U_{i,j^*}\right)+M\varpi(\omega_{i,k})\\
      &\leq U_{i,j^*}+U_{j^*,k}+C\delta_T (1+L)\varpi(\omega_{i,k})+M\varpi(\omega_{i,k}).
    \end{align}
    Then, we control $U_{i,j^*}$ in \eqref{eq:2} using \eqref{eq:39}
    with $(i,j,k)=(i,j^*-1,j^*)$,
    \begin{equation}
      \label{eq:Uj*}
      U_{i,k}\leq
      U_{i,j^*-1}+U_{j^*,k}+\delta_T(1+L)(\varpi(\omega_{i,j^*})+\varpi(\omega_{i,k}))+M\varpi(\omega_{i,j^*})+M\varpi(\omega_{i,k}).
    \end{equation}
    We now applying the induction hypothesis to $U_{i,j^*-1}$, 
    $U_{j^*,k}$ in \eqref{eq:Uj*}, and we use Global Hypothesis~\ref{hyp:4}
    to get
    \begin{equation}
      \label{eq:case1}
      U_{i,k}\leq 
      \kappa L\varpi(\omega_{i,k})+2\delta_T(1+L)\varpi(\omega_{i,k})+2M\varpi(\omega_{i,k}).
    \end{equation}
    Thus, with $L$ given by \eqref{eq:L}, $U_{i,k}\leq L\varpi(\omega_{i,k})$.

    In the second case, when $j^*=k$, we use \eqref{eq:39} with
    $j=k-1$ and \eqref{eq:defO} to get
    \begin{equation}
      \label{eq:3}
      U_{i,k}\leq U_{i,k-1}+\delta_T(1+L)\varpi(\omega_{i,k})+M\varpi(\omega_{i,k}).
    \end{equation}
    Thus, applying the induction hypothesis in \eqref{eq:3} to $U_{i,k-1}$,
    \begin{equation}
      \label{eq:4}
      U_{i,k}\leq \frac{\kappa L}{2}\varpi(\omega_{i,k})+\delta_T(1+L)\varpi(\omega_{i,k})+M\varpi(\omega_{i,k}).
    \end{equation}
    Eq.~\eqref{eq:4} implies \eqref{eq:case1}. It follows from
    the first case that $U_{i,k}\leq L\varpi(\omega_{i,k})$ with the
    same constant $L$. This concludes the induction.
    
   Therefore, 
    the numerical scheme associated to $\phi$
    belongs to~$\cP_\pi[\phi,a,L]$.
    That $\cP_\pi[\phi,a]=\cP_\pi[\psi,a]$ is immediate from
    \eqref{eq:gal:1}.
\end{proof}

The next result is a direct consequence of the continuous time Davie lemma 
\cite[Lemma~10]{brault2}. 

\begin{lemma}[Uniform control on D-solutions]
    \label{lem:1}
    Consider $\phi\in\cA[\delta,M]$.
    Assume that for some $A>0$, $y\in\cP[\phi,a,A]$.
    Then $y\in\cP[\phi,a,L]$ with $L$ given by \eqref{eq:L}.
    Therefore, $\cP[a,\phi]=\bigcup_{A\leq L}\cP[a,\phi,A]$.
\end{lemma}

\begin{notation}[Projection and interpolation]
    \label{not:5}
Let $\pi$ and $\sigma$ be two partitions of $\TT$ with $\sigma\subset \pi$. 
Any path $y$ in $\cP_\sigma[\phi,a]$ or in $\cP[\phi,a]$ 
is naturally projected onto $\Set{y_{t_i}}_{i=0}^n$ in $\cP_\pi[\phi,a]$.
Conversely, any element $y\in\cP_\pi[\phi,a]$ is extended through 
a linear interpolation as an element of $\cC([0,T],\uV)$. Again, we still
denote this element by~$y$.
\end{notation}

Using the above convention on projection and extension, we endow 
$\cP_\pi[\phi,a]$ with the uniform norm $\normsup{\cdot}$.
The proofs of the next lemmas are then immediate.

\begin{lemma}[Convergence]
    \label{lem:convergence}
    Let $K\geq 0$.
    Let $\Set{y^\pi}_{\pi}$ be a sequence
    of paths in $\cP_\pi[\phi,a,K]$ and $y\in\cC([0,T],\uV)$ 
    such that $y^\pi$ converges in $\normsup{\cdot}$ to $y$.
    Then $y\in\cP[\phi,a,K]$.
\end{lemma}
\begin{lemma}[Convergence II]
    \label{lem:convergence2}
    Let us consider $K,M\geq 0$.
    For each $n\in\NN$, let us consider $\phi^n\in\cA[\delta,M]$, $a^n\in\uV^n$
    and $y^n\in\cP[a^n,\phi^n,K]$.
    Let $\phi\in\cA[\delta,M]$ and $a\in\uV$. Assume
    that for some path $y\in\cC(\TT,\uV)$, 
    \begin{equation*}
	d_\infty(\phi^n,\phi)+\abs{a^n-a}+\normsup{y^n-y}\xrightarrow[n\to\infty]{}0.
    \end{equation*}
    Then $y\in\cP[\phi,a,K]$.
\end{lemma}

\subsection{From discrete to continuous functionals on D-solutions}

In this section, we construct functionals on $\cP_\pi[\phi,a]$ and thus
on $\cP[\phi,a]$ using a limit argument. These functionals are to be seen 
as integrals that are defined only on D-solutions, unlike Young or rough integrals.

\begin{proposition}
    \label{prop:2}
Let $\phi\in\cA[\delta,M]$ and $\pi$ be a partition of $\TT$. Recall
that $\widehat{\phi}$ is defined by~\eqref{eq:def:1}.
Let us set for $i,j\in\pi^2_+$, 
\begin{equation*}
    \Phi^\pi_{i,j}(y)\eqdef\sum_{k=i}^{j-1} \widehat{\phi}_{k+1,k}(y_{k})\text{ for }
    y=\Set{y_{i}}_{i=0,\dotsc,n}\in\uV^{n+1}.
\end{equation*}
For $y\in\cP_\pi[\phi,a,K]$, 
\begin{gather}
    \label{eq:37}
    \abs{\Phi^\pi_{i,j}(y)-\widehat{\phi}_{j,i}(y_i)}\leq 
   A \varpi(\omega_{i,j})\text{ for any }(i,j)\in\rpi^2
   \\
   \label{eq:38}
   \text{ with }A\eqdef \frac{2(\delta_T(1+K)+M)}{1-\varkappa}.
 \end{gather}
\end{proposition}

\begin{remark}
    \label{rem:5}
    We saw in Lemma~\ref{lem:2} that the numerical scheme $y^\pi$ associated to $\phi$ 
    with $y^\pi_0=a$ belongs to $\cP_\pi[\phi,a,L]$ with $L$ given by \eqref{eq:L}. 
    Therefore, 
    from the very construction of $y^\pi$: $y^\pi_{j}=a+\Phi_{0,j}^\pi(y^\pi)$.
\end{remark}

\begin{proof}
From the very definition of $\Phi^\pi$, 
\begin{equation}
    \label{eq:20}
\Phi_{i,j}^\pi(y)+\Phi_{j,k}^\pi(y)=\Phi_{i,k}^\pi(y)
\text{ for }(i,j,k)\in\rpi^3,
\end{equation}
meaning that $\Phi^\pi$ is additive on the partition $\pi$.

For any $(r,s,t)\in\rTT^3$ and $a\in\uV$,
\begin{equation*}
    \diff\phi_{t,s,r}(a)=
    \phi_{t,s}(\phi_{s,r}(a))-\phi_{t,s}(a)
    =
    \widehat{\phi}_{s,r}(a)
    +
    \widehat{\phi}_{t,s}(a+\widehat{\phi}_{s,r}(a))
    -\widehat{\phi}_{t,r}(a).
\end{equation*}
Thus, for $(i,j,k)\in\rpi^3$, 
\begin{equation}
    \label{eq:20bis}
    \widehat{\phi}_{k,j}(y_j)
    +\widehat{\phi}_{j,i}(y_i)
    -\widehat{\phi}_{k,i}(y_i)
    =
    \widehat{\phi}_{k,j}(y_j)
    -\widehat{\phi}_{k,j}(y_i+\widehat{\phi}_{j,i}(y_i))
    +\phi_{k,j,i}(y_i).
\end{equation}
Note that $y_j-y_i-\widehat{\phi}_{j,i}(y_i)=y_j-\phi_{j,i}(y_i)$.
Since $\phi\in\cA[\delta,M]$ and $y\in\cP_\pi[\phi,a,K]$,
\eqref{eq:defO}, \eqref{eq:def:1}, \eqref{eq:def:3} yield
\begin{multline}
    \label{eq:19}
    \abs{\widehat{\phi}_{k,j}(y_j)
    +\widehat{\phi}_{j,i}(y_i)
    -\widehat{\phi}_{k,i}(y_i)
}
\leq 
\varpi(\omega_{i,j})\delta_T\normO{\widehat{\phi}}\Paren*{1+K}
+M\varpi(\omega_{i,k})
\\
\leq (\delta_T(1+K)+M)\varpi(\omega_{i,k}),
\end{multline}
because $\normO{\widehat{\phi}}\leq 1$.
Combining \eqref{eq:20} with \eqref{eq:19}
implies that $V^\pi_{i,j}\eqdef\abs{\Phi_{i,j}^\pi(y)-\widehat{\phi}_{j,i}(y_i)}$
satisfies
\begin{equation*}
    V_{i,k}^\pi\leq V_{i,j}^\pi+V_{j,k}^\pi+(\delta_T(1+K)+M)\varpi(\omega_{i,k}).
\end{equation*}
Hence, \eqref{eq:37} stems from the Davie lemma \cite[Lemma~9]{brault2}.
\end{proof}
\begin{notation}
    \label{not:3}
    For a partition $\pi=\Set{t_i}_{i=0,\dotsc,n}$ of $\TT$, we set
    \begin{equation}
	\label{eq:mu}
	\mu_{s,t}(\pi)\eqdef \sup_{[t_i,t_{i+1}]\in \pi\cap [s,t]}
\frac{\varpi(\omega_{t_i,t_{i+1}})}{\omega_{t_i,t_{i+1}}}.
    \end{equation}
\end{notation}

\begin{remark}
    \label{rem:3:bis}
    With Remark~\ref{rem:3}, $\mu_{s,t}(\pi)\to 0$ when $\mesh{\pi}\to 0$.
\end{remark}

Let us consider a partition $\pi=\Set{t_i}_{i=0}^n$ of $\TT$. 
Using a linear interpolation,
$\Phi^\pi_{s,t}(y)$ is naturally extended from $\rpi^2$
to $\TT^+_2$.
Therefore, we extend to $\rTT^2$ the family $\Phi^\pi$ as functionals on $\cP[\phi,a]$
or on $\cP_\sigma[\phi,a]$ with $\pi\subset\sigma$.

\begin{corollary}[Consistency]
    \label{cor:3}
    Assuming Hypothesis~\ref{hyp:time} and $\phi\in\cA[\delta,M]$, 
    there exists $\Phi:\cP[\phi,a]\to\cC([0,T],\uV)$ such that 
    for any partition $\pi$ of $\TT$, any $y\in\cP[\phi,a,K]$ and any $K\geq 0$, 
    \begin{gather}
	\label{eq:34}
	\abs{\Phi_{s,t}(y)-\widehat{\phi}_{t,s}(y_s)}\leq A\varpi(\omega_{s,t}),
	\ \forall (s,t)\in\rTT^2,\\
	\label{eq:35}
	\abs{\Phi_{s,t}(y)-\Phi^\pi_{s,t}(y)}\leq A \mu_{s,t}(\pi)\omega_{s,t},
	\ \forall (s,t)\in\rTT^2,\\
	\label{eq:chasles}
	\text{and }
    \Phi_{r,s}(y)+\Phi_{s,t}(y)=\Phi_{r,t}(y)\text{ for }(r,s,t)\in\rTT^3,
    \end{gather}
    with $A$ given by \eqref{eq:38}. Condition \eqref{eq:35} means that $\Phi^\pi$ is 
    \emph{consistent}.
\end{corollary}

\begin{remark}
    This result does not claim that $\cP[\phi,a]\neq\emptyset$. 
    When $\uV$ is finite dimensional, 
    the Ascoli-Arzelà theorem and thus Lemma~\ref{lem:convergence} apply: 
    the equi-continuity and boundedness of $\Set{y^\pi}_\pi$
    with $y^\pi\in\cP_\pi[\phi,a,L]$ is a direct consequence of~\eqref{eq:def:3} and~\eqref{eq:1disc}. 
    When $\uV$ is infinite dimensional, we
    discuss this point in Section~\ref{sec:generic}.
\end{remark}

\begin{proof}
Let $\sigma$ and $\pi$ be two partitions such that $\pi\subset\sigma$. 
For $(s,t)\in\rpi^2$ and $y\in\cP[\phi,a,L]\subset\cP_\sigma[\phi,a,L]\subset\cP_\pi[\phi,a,L]$
(using the identification of Notation~\ref{not:5}),  
\begin{multline*}
    \abs{\Phi_{s,t}^\sigma(y)-\Phi^\pi_{s,t}(y)}
   =
    \abs*{\sum_{[u,v]\in\pi}
    \left[\sum_{[u',v']\in\sigma\cap[u,v]}
     \widehat{\phi}_{v',u'}(y_{u'}) -\widehat{\phi}_{v,u}(y_{u})\right]}
\\
=
    \abs*{\sum_{[u,v]\in\pi}\Paren{\Phi^{\sigma\cap[u,v]}_{u,v}(y)
        -\widehat{\phi}_{v,u}(y_{u})}}
  \\
\leq \sum_{[u,v]\in\pi} A\varpi(\omega_{u,v})
\leq A\mu_{s,t}(\pi) \omega_{s,t}
\xrightarrow[\mesh{\pi}\to 0]{}0
\end{multline*}
for $A$ given by \eqref{eq:38}.
From this, it is easily deduced that $\Set{\Phi^\pi_{s,t}(y)}_{\pi}$
is a Cauchy sequence
et for any $(s,t)\in\rTT^2$ with respect to the nest of nested sequence of partitions. 
We set $\Phi_{s,t}(y)\eqdef \lim_{\mesh{\pi}\to 0} \Phi_{s,t}^\pi(y)$. 
For any partition $\pi$ , \eqref{eq:35}
is satisfied and so is \eqref{eq:34} by taking $\pi=\Set{0,s,t,T}$.
We then set $\Phi_t(y)\eqdef\Phi_{0,t}(y)$. The Chasles relation \eqref{eq:chasles}
is satisfied because $\Phi^\pi$ satisfies the discrete Chasles relation \eqref{eq:20}.
Combining \eqref{eq:chasles} and \eqref{eq:34}, $\Phi_{s,t}(y)$ is 
uniquely defined thanks to the Additive Sewing Lemma~(see \textit{e.g.} \cite{lyons98a,gub04} 
    or~\cite[Theorem~1, p.~25]{feyel} or~\cite[Lemma~4.2 p.~51]{friz14a}).
\end{proof}
\begin{proposition}
    \label{prop:3}
 We assume Hypothesis~\ref{hyp:time} and $\phi$ an almost flow in $\cA$.
    A path $y\in\cC(\TT,\uV)$ satisfies $y_{t}=\Phi_{0,t}(y)$, $t\in\TT$, 
     if and only if $y\in\cP[\phi,a]$. 
\end{proposition}

\begin{proof}
    If $y\in\cP[\phi,a]$, both $\Set{y_{s,t}\eqdef y_t-y_s}_{(s,t)\in\TT^+_2}$ and $\Set{\Phi_{s,t}}_{s,t\in\rTT^2}$
    are additive functionals satisfying $\abs{z_{s,t}-\widehat{\phi}_{t,s}(z_s)}\leq C\varpi(\omega_{s,t})$
    for $(s,t)\in\pi_+^2$.
    From the Additive Sewing Lemma (see \textit{e.g.} \cite{lyons98a,gub04} or
\cite[Theorem~1, p.~25]{feyel} or \cite[Lemma~4.2 p.~51]{friz14a}), they are equal. Conversely, 
    if $y=\Phi(y)$, then with~\eqref{eq:34}, $\abs{y_{s,t}-\widehat{\phi}_{s,t}(y)}
    \leq A\varpi(\omega_{s,t})$ meaning that $y\in\cP[\phi,a,A]$.
\end{proof}

\subsection{Stability and convergence of discrete approximations}

We recover the general principle that consistency and stability yield convergence, 
as well as existence and uniqueness.
For this, we need a stronger hypothesis on $\Phi^\pi$.
We will show in Section~\ref{sec:stable-almost-flows}
that this hypothesis is satisfied in presence of \emph{stable almost flows}, as defined in \cite{brault2}. 

\begin{hypothesis}[Stability]
\label{hyp:stability}
Let $\phi\in\cA[\delta,M]$. 
Let $\Phi$ and $\Set{\Phi^\pi}_{\pi}$ be the associated functionals
given in Corollary~\ref{cor:3} and Proposition~\ref{prop:2}.
Assume that for each partition $\pi$ of $\TT$,
$\Phi^\pi$ is Lipschitz continuous on $\cP_\pi[\phi,a,L]$ with a constant $\ell<1$ which is uniform in $\pi$.
\end{hypothesis}

Thanks to the Lipschitz inverse function theorem, 
Hypothesis~\ref{hyp:stability} implies that $\mathrm{Id}-\Phi^\pi$ is invertible 
with a bounded inverse which is uniform in $\pi$. This is \emph{stability}.
We use in Corollary~\ref{cor:stability} below such a property on perturbations.

We now give the rate of convergence of numerical scheme. 
Applied to YDE and RDE (see \cite{brault1,brault2}), we recover the already found rates of convergence:
\begin{itemize}
    \item In \cite{davie05a}, $\varpi(x)=x^{\gamma/p}$ for a vector field in $\cC^\gamma$ and $x$ of finite $p$-variation, 
	$2\leq p<3$ and $1+\gamma>p$, see Remarks~1 and~3. Our estimate is a upper bound for the right-hand side of (9), namely 
    a rate of $\gamma/p-1$.
\item In \cite[Theorem~10.3.3]{friz}, a high order expansion of order $n$ for a rough path of finite $p$-variation, $2\leq p<3$,
    is given with $\varpi(x)=x^{(n+1)/p}$ for a vector field of class $\cC^{\gamma}$, $\gamma>p$ and $n=\lfloor \gamma\rfloor\geq\lfloor p\rfloor$.  The rate of convergence is $(n+1)/p-1$.
\item In \cite[Sect.~5, p.1789]{lejay10a}, for YDE ($1\leq p<2$) with a vector field of class $\cC^{\gamma}$, $1+\gamma>p$, 
    the rate of convergence is $2/p-1$ with $\varpi(x)=x^{2/p}$.
\end{itemize}

\begin{proposition}[Rate of convergence of approximations]
    \label{prop:convergence}
    Assume Hypothesis~\ref{hyp:stability} on stability.
    For each partition $\pi$ of $\TT$, 
    let $y^\pi$ be the numerical scheme associated to~$\phi$ with respect
    to~$\pi$ (see Definition~\ref{def:numerical-scheme}).

    Then, for any $y\in\cP[\phi,a,L]$, 
    \begin{equation}
	\label{eq:36b}
	\normsup{y^\pi-y}\leq \frac{A}{1-\ell}\mu_{0,T}(\pi)\omega_{0,T},
    \end{equation}
    where $\mu$ is defined in \eqref{eq:mu} and $A$ defined by
    \eqref{eq:38} with $K=L$.
    Besides, $\Set{y^\pi}_{\pi}$ is a Cauchy sequence with respect to $\normsup{\cdot}$ as $\mesh{\pi}\to 0$ with 
    \begin{equation}
	\label{eq:36}
	\normsup{y^\sigma-y^\pi}\leq \frac{2A}{1-\ell}\max\Set{\mu_{0,T}(\pi),\mu_{0,T}(\sigma)}\omega_{0,T}
    \end{equation}
    for any two partitions $\sigma$ and $\pi$ of $\TT$. 
    In consequence, $\cP[\phi,a]=\Set{y}$ with 
    $y=\lim_\pi y^\pi$.
\end{proposition}
\begin{proof} From Proposition~\ref{prop:3}, Definition~\ref{def:numerical-scheme} and Remark~\ref{rem:5},
    $y\in\cP[\phi,a,L]$ and $y^\pi\in\cP_{\pi}[\phi,a,L]$ are respectively fixed point solutions to 
    \begin{equation*}
	y_{s,t}=\Phi_{s,t}(y),\ \forall (s,t)\in\rTT^2
	\text{ and }y^\pi_{s,t}=\Phi_{s,t}^\pi(y^\pi),\ \forall (s,t)\in\rpi^2,
    \end{equation*}
    where $\Phi$ is given by Corollary~\ref{cor:3}.
    For $(s,t)\in\rpi^2$, as $\cP[\phi,A,L]\subset \cP_\pi[\phi,A,L]$ (recall Notation~\ref{not:5}), 
    \begin{equation*}
	y^\pi_{s,t}-y_{s,t}=\Phi^\pi_{s,t}(y^\pi)-\Phi^\pi_{s,t}(y)
	+\epsilon^\pi_{s,t}
	\text{ with }\epsilon^\pi_{s,t}\eqdef \Phi^\pi_{s,t}(y)-\Phi_{s,t}(y).
    \end{equation*}
    Hence, from Hypothesis~\ref{hyp:stability},
    \begin{equation*}
	\abs{y^\pi_{s,t}-y_{s,t}}\leq \ell\normsup{y^\pi-y}+\abs{\epsilon^\pi_{s,t}}.
    \end{equation*}
    With \eqref{eq:35} in Corollary~\ref{cor:3} and since $y^\pi_0=y_0$,
    \begin{equation*}
	\normsup{y^\pi-y}\leq \ell\normsup{y^\pi-y}+A\mu_{0,T}(\pi)\omega_{0,T}.
    \end{equation*}
    As the uniform Lipschitz constant of $\Phi^\pi$ satisfies $\ell<1$ from Hypothesis~\ref{hyp:stability}, this proves \eqref{eq:36b}.

    If $z,y\in\cP[\phi,a,K]$, then for any partition $\pi$ of $\TT$, 
    $\normsup{y-z}\leq \normsup{y-y^\pi}+\normsup{z-y^\pi}$,
    so that $y=z$. This proves uniqueness since $\mu_{0,T}(\pi)$ defined by \eqref{eq:mu}
    decreases to $0$ with the mesh of $\pi$.

    To prove \eqref{eq:36}, we consider first two nested partitions $\pi$ and $\sigma\subset \pi$.
    We proceed as above with $y$ replaced by $y^\sigma$. When $\sigma$ and $\pi$ are arbitrary partitions, 
    then there exists a partition $\tau$ such that $\tau\subset\sigma$ and $\tau\subset\pi$. 
    The triangle inequality yields~\eqref{eq:36}. That $\cP[\phi,a]$ is contains
    only the limit of $\Set{y^\pi}_{\pi}$ follows from Lemma~\ref{lem:convergence}, 
    since $y^\pi\in\cP_\pi[\phi,a,L]$ from Lemma~\ref{lem:2}.
\end{proof}

\subsection{Stable almost flows and continuity}

\label{sec:stable-almost-flows}

We give now a sufficient condition to ensure Hypothesis~\ref{hyp:stability}.
The notion of stable almost flow was introduced in \cite{brault2}.

\begin{notation}[Ratio bound]
    For $\generalnorm{\star}{\cdot}$ being either $\normsup{\cdot}$ or $\normlip{\cdot}$, we define
    for \mbox{$\phi:\rTT^2\to\cC(\uV,\uV)$}, 
    \begin{equation*}
	\generalnorm{\star\div\varpi}{\phi}\eqdef \sup_{\substack{(r,t)\in\rTT^2\\r\neq t}}
    \frac{\generalnorm{\star}{\phi_{t,r}}}{\varpi(\omega_{r,t})}
    \text{ and }
	\generalnorm{\star\div\varpi}{\diff\phi}\eqdef \sup_{\substack{(r,s,t)\in\rTT^3\\r\neq t}}
    \frac{\generalnorm{\star}{\diff\phi_{t,s,r}}}{\varpi(\omega_{r,t})}.
    \end{equation*}
\end{notation}

\begin{definition}[Stable almost flow]
    A \emph{stable almost flow} is an almost flow
    $\phi\in\cA[\delta,M]$  with $\normlip{\phi_{t,s}-\id}\leq \delta_T$
    which satisfies
    \begin{gather}
	\label{eq:saf:1}
	\normlipvarpi{\diff \phi}<+\infty, 
    \end{gather}
    as well as the \emph{4-points control} 
    \begin{multline}
	\label{eq:4pc}
	\abs{\phi_{t,s}(a)-\phi_{t,s}(b)-\phi_{t,s}(c)+\phi_{t,s}(d)}
	\\
	\leq 
	\widecheck{\phi}_{t,s}\Paren[\big]{\abs{a-b}\vee\abs{c-d}}
	\times \Paren[\big]{\abs{a-c}\vee\abs{b-d}}+(1+\delta_T)\abs{a-b-c+d},
    \end{multline}
    where for any $\alpha\geq 0$, 
    \begin{equation*}
	\widecheck{\phi}_{t,s}(\alpha\varpi(\omega_{r,s}))\leq \phi^{\circledast}(\alpha)\varpi(\omega_{r,t}),
	\ \forall (r,s,t)\in\rTT^3,
    \end{equation*}
    for $\phi^{\circledast}\geq 0$ that depends on $\alpha$ and $\omega_{0,T}$.
    Let us denote by $\cSA$ the subset of $\cA$ of stable almost flows.
\end{definition}
\begin{proposition}
    \label{prop:stability}
    Let $\phi\in\cSA$ be a stable almost flow. 
    Then the corresponding functional $\Phi^\pi$ given by Corollary~\ref{cor:3} 
    satisfies
    \begin{multline*}
	\sup_{(t_i,t_j)\in\rpi^2}
	\abs{\Phi^\pi_{i,j}(y)-\Phi^\pi_{i,j}(z)}
	\leq \ell_T\normsup{y-z}
	\\
	\text{ with }
	\ell_T \eqdef
	\delta_T+\frac{\normlipvarpi{\diff \phi}+(1+\delta_T)(2+\delta_T)+\phi^\circledast(K)
}{1-\varkappa}\varpi(\omega_{0,T}) 
    \end{multline*}
    for any $y,z\in\cP_\pi[\phi,a,K]$. In particular, $\ell_T\xrightarrow[T\to0]{}0$.
\end{proposition}

\begin{proof} Consider $y,z\in\cP_\pi[\phi,a,K]$. Let us set
    \begin{equation}
	\label{eq:24}
	V_{i,j}
	\eqdef
	\Phi_{i,j}^\pi(y)-\widehat{\phi}_{j,i}(y_i)
	-\Phi_{i,j}^\pi(z)+\widehat{\phi}_{j,i}(z_i).
    \end{equation}
    As $\phi_{j,i}=\id+\widehat{\phi}_{j,i}$, we rewrite \eqref{eq:24} as
    \begin{equation*}
	V_{i,j}=
\Phi_{i,j}^\pi(y)-\Phi_{i,j}^\pi(z)
+\phi_{j,i}(y_i) -\phi_{j,i}(z_i)+y_i-z_i.
    \end{equation*}
    Using \eqref{eq:20} in the proof of Proposition~\ref{prop:2}, 
    \begin{multline*}
	V_{i,j}+V_{j,k}-V_{i,k}
	=
\phi_{j,i}(y_i) -\phi_{j,i}(z_i)+y_j-z_j
+
\phi_{k,j}(y_j) -\phi_{k,j}(z_j)
-\phi_{k,i}(y_i) +\phi_{k,i}(z_i)
\\
=
\phi_{j,i}(y_i)
-\phi_{j,i}(z_i)
+y_j-z_j
+\diff\phi_{k,j,i}(y_i)
-\diff\phi_{k,j,i}(z_i)
\\
+\phi_{k,j}(y_j)-\phi_{k,j}(\phi_{j,i}(y_i))
-\phi_{k,j}(z_j)+\phi_{k,j}(\phi_{j,i}(z_i)).
    \end{multline*}
    Since $\phi$ is stable almost flow, 
    the 4-points control \eqref{eq:4pc} on $\phi$ yields
    \begin{multline*}
	\abs{V_{i,j}+V_{j,k}-V_{i,k}}
	\\
	\leq 
	\widecheck{\phi}_{k,j}\Paren[\big]{
	    \abs{y_{j}-\phi_{j,i}(y_i)}\vee\abs{z_{j}-\phi_{j,i}(z_i)}
	}
	\times
	\Paren[\big]{
	    \abs{y_j-z_j}\vee
	    \abs{\phi_{j,i}(y_i)-\phi_{j,i}(z_i)}
	}
	\\
	+(2+\delta_T)\abs{y_j-z_j-\phi_{j,i}(y_i)+\phi_{j,i}(z_i)}
	+\normlipvarpi{\diff \phi}\abs{y_i-z_i}\varpi(\omega_{i,k})
	\\
	\leq B\normsup{y-z}\varpi(\omega_{i,k})
    \end{multline*}
    for $(i,j,k)\in\rpi^+$, where
    \begin{equation}
	\label{eq:B}
	B\eqdef\normlipvarpi{\diff
          \phi}+(1+\delta_T)(2+\delta_T)+(1\vee\delta_T) \phi^\circledast(K).
    \end{equation}
    Moreover, $V_{i,i}=V_{i,i+1}=0$. From the Davie lemma (Lemma~9 in~\cite{brault2})
    with $U_{i,j}\eqdef\abs{V_{i,j}}$,  
    \begin{equation*}
	\abs{V_{i,j}}\leq \frac{2B}{1-\varkappa}\normsup{y-z}\varpi(\omega_{i,j}),\ \forall (i,j)\in\rpi^+.
    \end{equation*}
    Hence, 
    \begin{multline*}
	\abs{\Phi^\pi_{i,j}(y)-\Phi^\pi_{i,j}(z)}
	\leq \normlip{\widehat{\phi}_{j,i}}\normsup{y-z}+\frac{B}{1-\varkappa}\normsup{y-z}\varpi(\omega_{i,k})
	\\
	\leq 
	\Paren*{\delta_T+\frac{2B}{1-\varkappa}\varpi(\omega_{0,T})} \normsup{y-z}.
    \end{multline*}
    This proves the result.
\end{proof}

\begin{corollary}
    \label{cor:5}
    Let $\phi\in\cSA$ be a stable almost flow. Then for $T$ small enough, 
    Hypothesis~\ref{hyp:stability} is satisfied and thus~\eqref{eq:36} holds true.
\end{corollary}

\subsection{Continuity results for stable almost flows}
\label{sec:continuity}

The next proposition is a discrete version of \cite[Proposition~10]{brault2}
on the distance between two numerical schemes, one associated to a stable almost flow.
The proof is close to the one of Proposition~\ref{prop:stability}. The next 
result is the key to prove generic conditions.

\begin{notation}[Distance on almost flows]
    \label{not:4}
    For $\phi,\psi\in\cA[\delta,M]$, we define 
using~\eqref{eq:def:1}, \eqref{eq:def:5} and \eqref{eq:phitsr}, 
\begin{equation*}
    \label{eq:def:4}
    d_{\cA}(\phi,\psi)\eqdef 
    \max\Set*{d_\infty(\phi,\psi),\normO{\phi-\psi},\normsupvarpi{\diff\phi-\diff\psi}}.
\end{equation*}

\end{notation}

\begin{proposition}
    \label{prop:stability:schemes}
Let $\phi\in\cSA\cap\cA[\delta,M]$ 
be a stable almost flow and $\psi\in\cA[\delta,M]$ be an almost flow.

Consider a partition $\pi=\Set{t_i}_{i=0}^n$ of $\TT$. 
Let $y^\pi$ and $z^\pi$ be the numerical schemes associated to $\phi$ and 
$\psi$ with $y^\pi_0=a$ and $z^\pi_0=b$.

Then there exists a time $T$ small enough and  constants $C$ and $C'$ that depend only on 
$L$ given by \eqref{eq:L}, 
$\phi^{\circledast}(L)$, $\delta$, $\varkappa$ and $\generalnorm{\Lip,\varpi}{\phi}$
such that 
\begin{gather*}
    \abs{y^\pi_j-\phi_{j,i}(y^\pi_i)-z^\pi_j+\psi_{j,i}(z^\pi_i)}
    \leq Cd_\cA(\phi,\psi)\varpi(\omega_{i,j}),\ \forall (i,j)\in\rpi^+,
    \\
    \normsup{y^\pi-z^\pi}\leq Cd_{\cA}(\phi,\psi)+C'\abs{a-b}.
\end{gather*}
\end{proposition}
\begin{proof}
    The next proposition is a discrete version of \cite[Proposition~10]{brault2}.
    Set 
    \begin{equation*}
	U_{i,k}\eqdef y^\pi_i-\phi_{k,i}(y^\pi_i)-z^\pi_i+\psi_{k,i}(z^\pi_i).
    \end{equation*}
    For  $i=0,\dotsc,n$, $U_{i,i}=U_{i,i+1}=0$ from the definition of $y^\pi$ and $z^\pi$.

Set $\alpha_{j,i}\eqdef\phi_{j,i}-\psi_{j,i}$ 
and $\alpha_{k,j,i}\eqdef\phi_{k,j,i}-\psi_{k,j,i}$. Assume that 
for any $(i,j,k)\in\rpi^3$, 
\begin{gather*}
    \abs{\alpha_{k,j,i}(z^\pi_i)}\leq \epsilon_1\varpi(\omega_{i,k}),\\
    \osc(\alpha_{k,j},\abs{z_j-\phi_{j,i}(z_i)})\leq \delta_T\epsilon_2(1+L)\varpi(\omega_{i,k})
    \text{ and }
    \abs{\alpha_{j,i}(z^\pi_i)}\leq \epsilon_3.
\end{gather*}

    With \eqref{eq:39} and  since $y^\pi\in\cP_\pi[\phi,a,L]$ and $z^\pi\in\cP_\pi[\psi,b,L]$, 
    the 4-points control~\eqref{eq:4pc} on $\phi$ yields 
    \begin{multline*}
	\abs{U_{i,k}}\leq \abs{U_{i,j}}+(1+\delta_T)\abs{U_{j,k}}
	+\abs{\phi_{i,j,k}(z^\pi_i)-\psi_{i,j,k}(z^\pi_i)}\\
	+\varpi(\omega_{i,k})
	\Big(
	    \phi^\circledast(L)(1+\delta_T)\normsup{y^\pi-z^\pi}
	    +\phi^\circledast(L)\abs{\phi_{j,i}(z^\pi_i)-\psi_{j,i}(z^\pi_i)}
	    \\
	    +\delta_T\epsilon_2(1+L)\varpi(\omega_{i,k})
	    +\normlipvarpi{\diff \phi}\normsup{y^\pi-z^\pi}
	)
	\\
	\leq
\abs{U_{i,j}}+(1+\delta_T)\abs{U_{j,k}}
+(N+N'\normsup{y^\pi-z^\pi})\varpi(\omega_{i,k})
\end{multline*}
where   
\begin{align*}
    N&\eqdef\Paren*{\epsilon_1+\delta_T(1+L)\epsilon_2+\phi^\circledast(L)\epsilon_3}
    \leq (1+\delta_T L^\gamma+\phi^\circledast)d_{\cA}(\phi,\psi)
    \\
\text{ and }
N'&\eqdef\Paren*{\phi^\circledast(L)(1+\delta_T)+\normlipvarpi{\diff \phi}}.
    \end{align*}
    The Davie Lemma \cite[Lemma 9]{brault2} implies that 
    \begin{gather}
	\label{eq:41}
	\abs{U_{i,k}}\leq ND\varpi(\omega_{i,k})+N'D\normsup{y^\pi-z^\pi}\varpi(\omega_{i,k})
	\ \forall (i,k)\in\pi_+^2
	\\
	\notag
	\text{with } 
	D\eqdef 
    \frac{2+\delta_T}{1-\varkappa(1+\delta_T)^2-\delta_T}.
    \end{gather}
    Thus, 
    \begin{equation*}
	\normsup{y^\pi-z^\pi}\leq ND\varpi(\omega_{0,T})+N'D\normsup{y^\pi-z^\pi}\varpi(\omega_{0,T})
	+(1+\delta_T)\abs{a-b}.
    \end{equation*}
    Assuming that $T$ is small enough so that $N'D\varpi(\omega_{0,T})<1$, 
    \begin{equation}
      \label{eq:40}
      \normsup{y^\pi-z^\pi}\leq \frac{ND}{1-N'D\varpi(\omega_{0,T})}+\frac{1+\delta_T}{1-N'D\varpi(\omega_{0,T})}\abs{a-b}.
    \end{equation}
    Injecting \eqref{eq:40} into \eqref{eq:41} leads to the result.
\end{proof}
\begin{notation}[Perturbations]
    Let $\cE$ be the family of elements $\epsilon\in\cO(\uV,\uV)$ such that 
for some parameter~$\eta\geq 0$, 
    \begin{equation}
	\label{eq:42}
	\normO{\epsilon}\leq \eta\text{ and }\generalnorm{\infty\div\varpi}{\epsilon}\leq \eta.
    \end{equation}
    We denote by $\normE{\epsilon}$ the minimal value of $\eta$ for which \eqref{eq:42} holds.
    An element of $\cE$ is called a \emph{perturbation} \cite{brault1}.  
\end{notation}
\begin{notation}[Perturbed numerical schemes]
    Given an almost flow $\phi\in\cA[\delta,M]$, a perturbation $\epsilon\in\cE$, a starting point $a\in\uV$
    and a partition $\pi=\Set{t_i}_{i=0}^n$, the \emph{perturbed numerical scheme} associated $(\phi,\epsilon)$ is
    $z^\pi_{k+1}=\phi_{k+1,k}(z^\pi_k)+\epsilon_{k+1,k}(z^\pi_k)$ with $z^\pi_0=a$. A perturbed numerical scheme 
    solves $z^\pi_{i,j}=\Phi^\pi_{i,j}(z^\pi)+E_{i,j}$ with~$E_{i,j}=\sum_{k=i}^{j-1} \epsilon_{i+1,i}(z_i^\pi)$.
\end{notation}

In the context of numerical analysis, a perturbation $\epsilon$ corresponds for example to \emph{round-off errors}
while the choice of an almost flow correspond to \emph{truncation error}. 

\begin{corollary}[Stability of perturbed numerical schemes]
    \label{cor:stability}
    Let $\phi$ be a stable almost flow and $\epsilon\in\cE$.  
    Then there exists a constant $K$ depending on $\varpi$, $\omega_{0,T}$, $M$ and $\delta$ such that 
    for any partition $\pi$ of $\TT$,
    \begin{equation}
	\label{eq:43}
	\normsup{y^\pi-z^\pi}\leq K\normE{\epsilon}, 
    \end{equation}
    where  $y^\pi$ is the numerical scheme associated to $\phi$ and $z^\pi$ is the perturbed numerical scheme 
    associated to $(\phi,\epsilon)$.
\end{corollary}
\begin{proof} 
    From \cite{brault1}, $\psi_{t,s}\eqdef\phi_{t,s}+\epsilon_{t,s}$ is an almost flow, 
    yet not necessarily a stable one, that belongs to $\cA[\delta(1+\eta),M+(2+\delta_T)\eta]$. Moreover, 
    \begin{equation*}
	d_{\cA}(\psi,\phi)\leq \normE{\epsilon}\max\Set{(2+\delta_T),\varpi(\omega_{0,T})}.
    \end{equation*}
    Inequality~\eqref{eq:43} stems from Proposition~\ref{prop:stability:schemes}.
\end{proof}
\section{Generic properties of flows}

\label{sec:generic}

\subsection{The generic property}

Related to differential equations, a \emph{generic property} is a property which holds for
\textquote{almost all} (in the sense of Baire) vector fields and starting points. A precise
description relies on the notion of residual set. The study of generic properties
to differential equations have started with W.~Orlicz~\cite{orlicz}. Many results are exposed in~\cite{myjak}.

\begin{definition}[Residual set]
    A set $\cN$ in a complete metric space $\cM$ is \emph{residual} if 
    its $\cM\setminus\cN$ is of Baire first category.
\end{definition}
\begin{definition}[Generic property]
    A property is said to be \emph{generic} if it is true on a residual set. 
\end{definition}
We now state our main result, which is an adaptation of the ones in \cite{lasota,deblasi83,myjak}
to our setting. It relies on the following lemma. 

\begin{lemma}[{A. Lasota \& J.A. Yorke, \cite[Lemma~1.2]{myjak}}]
    \label{lem:generic} 
    Let $\cM$ be a complete metric space with a dense subset $\cN$. Assume that there exists $\Theta:\cM\to\RR_+$
    such that $\Theta(x)=0$ for any $x\in\cN$ and $\Theta$ is continuous at any $x\in\cN$. 
    Then \mbox{$\Set{x\in\cM\given \Theta(x)=0}$} is residual in $\cM$.
\end{lemma}

\begin{hypothesis}
    \label{hyp:3}
    We consider a complete metric space $(\cQ,d)$ with a dense subspace~$\cR$.
    There exists a continuous mapping from $(\cQ,d)$ to $\phi[f]\in(\cA[\delta,M],d_\cA)$
    ($d_\cA$ is defined in Notation~\ref{not:4}) which transforms $f\in\cQ$ into $\phi[f]\in\cA[\delta,M]$ 
    and such that $\phi[f]\in\cSA[\delta,M]$ for any $f\in\cR$.
\end{hypothesis}

\begin{theorem}[Generic property of existence, uniqueness and convergence]
    \label{thm:generic}
    Under Hypothesis~\ref{hyp:3}, 
    existence, uniqueness of D-solution and convergence of numerical schemes are generic properties.
    More precisely, let $\cN$ be the subset of $\cM=\uV\times\cQ$ 
    such that the numerical schemes $y^\pi\in\cC(\TT,\uV)$ associated to $\phi[f]$ with $y^\pi_0=a$  
    converges uniformly with respect to $\pi$ to some $y\in\cC(\TT,\uV)$.
    Then $\cN$ is a residual set in $\uV\times\cQ$
    and $y\in\cP[\phi[f],a]$. In addition, the subset $\uV\times\cR$ of $\cM$ such that $\cP[\phi[f],a]$
    contains only one point is a residual set.
\end{theorem}

\begin{proof}
Let us define for $a\in\uV$ and $f\in\cQ$, 
\begin{equation*}
    \Theta((a,f))\eqdef\limsup_{\substack{\pi,\sigma\\ \mesh{\pi},\mesh{\sigma}\to 0}} \normsup{y^\pi[f,a]-y^\sigma[f,a]},
\end{equation*}
where $y^\pi[f,a]$ is the numerical scheme associated to $\phi[f]$ with $y^\pi[f,a]_0=a$.
With Proposition~\ref{prop:3}, $\Theta((a,f))=0$ for any $(a,f)\in\uV\times\cR$.

Let $\Set{(a_k,f_k)}_{k\geq 0}$ be a sequence of elements of $\cM$ converging to $(a,f)\in\uV\times\cR$.
By Hypothesis~\ref{hyp:3}, $\phi[f_k]\in\cA[\delta,M]$ while $\phi[f]\in\cSA[\delta,M]$.

By the triangle inequality, 
\begin{multline*}
    \normsup{y^{\pi}[f_k,a_k]-y^\sigma[f_k,a_k]}
    \leq 
    \normsup{y^{\pi}[f_k,a_k]-y^\pi[f,a]}
    +\normsup{y^\pi[f,a]-y^\sigma[f,a]}
    \\
    +\normsup{y^{\sigma}[f_k,a_k]-y^\sigma[f,a]}.
\end{multline*}
Using Corollary~\ref{cor:5} and Proposition~\ref{prop:stability:schemes}, 
\begin{multline*}
    \normsup{y^{\pi}[f_k,a_k]-y^{\sigma}[f_k,a_k]}
    \\
    \leq 
    2C\Paren*{d_\cA(\phi[f_k],\phi[f])+\abs{a-a_k}}
    +C'\max\Set{\mu_{0,T}(\pi),\mu_{0,T}(\sigma)},
\end{multline*}
for a constant $C$ which depends on $f$ but which is uniform in $\pi,\sigma$, and a constant $C'$
which is uniform on $\pi,\sigma$.
Thus, for any $\epsilon>0$, one may choose $k_0$ large enough such that for any $k\geq k_0$
$2C\Paren*{d_\cA(\phi[f_k],\phi[f])+\abs{a-a_k}}\leq \epsilon$  
as well as some $\eta$ such that when $\max\Set{\mesh{\pi},\mesh{\sigma}}<\eta$, 
$C'\max\Set{\mu_{0,T}(\pi),\mu_{0,T}(\sigma)}\leq \epsilon$. 
Therefore, for any $k\geq k_0$, 
$\Theta((a_k,f_k))\leq 2\epsilon$ and $\lim_k \Theta((a_k,f_k))=0$.
It follows that $\Set{\Theta((a,f))=0\given (a,f)\in\uV\times\cQ}$,
which contains $\uV\times\cR$, is residual in $\uV\times\cQ$. 

For the uniqueness, we replace $\Theta$ by 
\begin{equation*}
    \Theta((a,f))\eqdef \sup_{y,z\in \cP[\phi[f],a]}\normsup{y-z}.
\end{equation*}
Again by Proposition~\ref{prop:convergence}, $\Theta((a,f))=0$ for $(a,f)\in\uV\times\cR$. 
The proof is similar to the above one.
\end{proof}

\subsection{Application to RDE}

We consider the case of RDE $y_t=a+\int_0^t f(y_s)\vd\bx_s$, the result being similar for YDE.
The driving rough path lies above a path living in a Banach space $\uU$, while 
the solution $y$ lives in another Banach space $\uV$.

let us fix $2\leq p<3$.
We consider a $p$-rough path $\bx\eqdef(1,\bx^{(1)},\bx^{(2)})$ with respect to the control $\omega$
with values in $\RR\oplus \uU\oplus \uU^{\bigotimes 2}$, \cite[Definition 3.1.3]{lyons98a}.
This means that $\bx$ satisfies $\bx_{r,s}\otimes\bx_{s,t}=\bx_{s,t}$ for any $r\leq s\leq t\leq T$ and 
\begin{equation*}
    \normp{\bx}\eqdef \sup_{\substack{(s,t)\in\rTT^2\\s\neq t}}
    \Paren*{
    \frac{\abs{\bx^{(1)}_{s,t}}}{\omega_{s,t}^{1/p}}
    +\frac{\abs{\bx^{(2)}_{s,t}}}{\omega_{s,t}^{2/p}}
}<+\infty.
\end{equation*}

\begin{definition}[Lipschitz vector fields]
    For any $\gamma>0$,  a vector field $f:\uV\to L(\uU,\uV)$ is said to be a 
    $\Lip(\gamma)$-vector field (which we write $f\in\Lip(\gamma)$) if it is of class $\cCb^{\floor{\gamma}}$ with 
    \begin{equation*}
	\normhold{\gamma}{f}\eqdef
	\sum_{k=0,\dotsc,\floor{\gamma}}
    \normsup{\dD^k f}+ \normhold{\gamma-\floor{\gamma}}{\dD^k f}<+\infty,
    \end{equation*}
where $\normhold{\lambda}{f}\eqdef \sup_{x\not= y}\abs{f(x)-f(y)}/\abs{x-y}^\lambda$ is the $\lambda$-Hölder norm
for $0<\lambda\leq 1$.
\end{definition}

Fix $R\geq 0$ and $\gamma\geq p-1$. We define
\begin{equation*}
    \cU(R,\gamma)\eqdef \Set*{ f\in\Lip(\gamma)\given \normhold{\gamma}{f}\leq R}.
\end{equation*}
We use $\normhold{\gamma}{\cdot}$ as a norm on $\cU(R,\gamma)$.

For $f\in\cU(R,\gamma)$, the \emph{Davie approximation} is the family
\begin{equation}
    \label{eq:davie}
    \phi_{t,s}[f,\bx](a)=a+f(a)\bx^{(1)}_{s,t}+f^{(2)}(a)\bx^{(2)}_{s,t}
    \text{ for }a\in\uV\text{ and }(s,t)\in\rTT^2
\end{equation}
with $f^{(2)}(a)=\dD f(a)\cdot f(a)$.
When $1+\gamma>p$, $\phi[f,\bx]$ is an almost flow \cite{brault1}. When $\gamma>p$,  
then it is a stable almost flow \cite{brault2}.
A regularisation argument implies that when $1\leq \gamma\leq 3$, then $\cU(R,3)$ is dense into $\cU(R,\gamma)$.

\begin{lemma}
    Assume $\gamma>2$.  Let $f_n\in\cU(R,\gamma)$
    which converges to $f\in\cU(R,\gamma)$.
    Then $d_{\cA}(\phi[f_n,\bx],\phi[f,\bx])$ converges to $0$.
\end{lemma}

\begin{proof}
 A classical computation shows that when $\gamma>1$, 
 \begin{multline}
     \label{eq:21}
     \diff\phi_{t,s,r}[f](a)=
     \Paren*{
     f(a+f(a)\bx^{(1)}_{r,s}+f^{(2)}(a)\bx^{(2)}_{r,s})-f(a+f(a)\bx^{(1)}_{r,s})}\bx^{(1)}_{s,t}
     \\
     =\Paren*{f(a+f(a)\bx^{(1)}_{r,s}+f^{(2)}(a)\bx^{(2)}_{r,s})-f(a))}\bx^{(2)}_{s,t}
     \\
     +\int_0^1 \Paren*{\dD f(a+\tau f(a)\bx^{(1)}_{r,s})-\dD f(a)}f(a)\bx^{(1)}_{r,s}\otimes \bx^{(1)}_{r,t}.
 \end{multline}
 Thus, for a constant $M$ that depends only on $\normhold{\gamma}{f}$ and $\normp{\bx}$, 
 \begin{equation}
     \label{eq:22}
     \abs{\diff\phi_{t,s,r}[f](a)}
     \leq M \varpi(\omega_{r,t})
 \end{equation}
 with 
 \begin{equation}
     \label{eq:23}
    \varpi(x)=x^{(2+\gamma)/p} \text{ and } M\leq \normhold{\gamma}{f}^2\max\Set{\normp{\bx}^{1+\gamma},\normp{\bx}^{2+\gamma}}.
 \end{equation}
 For $\gamma>2$, we easily deduce from \eqref{eq:davie} and \eqref{eq:21} that $\phi[f_n,\bx]$ 
 converges to $\phi[f,\bx]$ with respect to $d_\cA$, up to changing $\gamma$ into $\gamma'<\gamma$.

 The result follows from straightforward computations.
\end{proof}

Combining the above results with Theorem~\ref{thm:generic} leads to the following result. 
The second points is obtained by applying a theorem of Kuratowski and Ulam \cite{kuratowski}
(See also Theorem~4.2 in \cite{deblasi83}).

\begin{corollary}
Existence, uniqueness and convergence of the numerical scheme related to the RDE
$y=a+\int_0^\cdot f(y_s)\vd \bx_s$ is generic with respect to $(a,f)\in\uV\times\cU(R,\gamma)$
when $\gamma\geq 2$. In addition, if $\uV$ is separable, then there exists a residual 
set $\cR$ in $\cU(R,\gamma)$ such that for an $f\in\cR$, there exists a residual 
set $\mathcal{V}[f]$ such that existence, uniqueness and convergence of the numerical 
scheme holds for $a\in\mathcal{V}[f]$.
\end{corollary}


\section{Flows of diffeomorphisms through Brownian flows}

\label{sec:br-flow}

Let $(A,\Sigma,\PP)$ be a probability space. In the following we note without necessarily specifying it, by
$\alpha$ some element of $A$.
Moreover, for an integer $k$, $\uL^k(A)$ denotes the space of random
variables $K$ such that the quantity $\normf{\uL^k}{K}=\EE(K^k)^{1/k}$ is finite.
In this section, the state space of the driving Brownian motion is $\uU=\RR^d$, 
while the state space of the solutions is $\uV=\RR^m$ for some $d,m\geq 1$.

\begin{hypothesis}
    \label{hyp:7}
    Let $\sigma:\RR^m\to L(\RR^d,\RR^m)$ be a continuous function in $\sigma\in\cCb^{1+\gamma}$
    for~$\gamma\in (0,1]$. 
\end{hypothesis}

Let $B$ be a $d$-dimensional Brownian motion on $(A,\Sigma,\PP)$. We consider the family of Itô SDE
\begin{equation}
    \label{eq:ito:1}
    X_t(a)=a+\int_0^t \sigma(X_s(a))\vd B_s\text{ for } t\geq 0,\ a \in \RR^m.
\end{equation}
Under Hypothesis~\ref{hyp:7} (even with $\gamma=0$), there exists a unique strong solution to~\eqref{eq:ito:1}.

\begin{notation}
    \label{not:ebm}
    An \textit{enhanced} (Itô) Brownian motion \cite[Sect.~3.2]{friz14a} is a rough paths~$\bB$ of 
    order $2$ decomposed as 
$\bB_{r,t}=1+B_{r,t}+\bB^{(2)}_{r,t}$ with 
\begin{equation*}
    \bB^{(2)}_{r,t}=\int_r^t B_{r,s}\otimes\dd B_s,\ \forall (r,t)\in\rTT^2.
\end{equation*}
We assume that $(A,\Sigma,\PP)$ carries $\bB$.
\end{notation}

The \emph{Davie approximation} is naturally defined as 
\begin{equation}
    \label{eq:ebm:davie}
    \phi_{t,s}[\sigma,\alpha](a)\eqdef a+\sigma(a)B_{s,t}(\alpha) +\dD\sigma(a)\cdot\sigma(a)\bB^{(2)}_{s,t}(\alpha)
    \text{ for } \alpha\in A.
\end{equation}
In \cite{brault1}, we saw that $\phi[\sigma,\alpha]$ is an almost flow when $\sigma\in\cCb^{1+\gamma}$.  
When $\sigma\in\cCb^{2+\gamma}$, $\phi$ is a stable almost flow. This latter case grants 
uniqueness of D-solutions as well as the existence of a Lipschitz flow.

Here, we consider a deterministic function $\sigma\in\cCb^{1+\gamma}\setminus\cCb^2$.
Actually, for any $\alpha$, Theorem~4.8 in \cite{davie05a} shows that for almost every $\alpha\in A$, 
there exists a vector field~$\sigma[\alpha]$ such that infinitely many D-solutions exist.
For such a choice, $\phi[\sigma(\alpha),\alpha]$ cannot be a stable almost flow. 
Therefore, we cannot expect that $\phi[\sigma,\alpha]$ is a stable almost flow 
for any pair $(\sigma,\alpha)$. Our main result states the existence of a Lipschitz flow,
but does not prove that $\phi[\sigma,\alpha]$ is a stable almost flow.

A series of well-known results of H.~Kunita state that 
$a\mapsto X_t(a,\alpha)$ defines a flow of diffeomorphisms for almost every $\alpha\in A$ (see below). Our main theorem 
states that under Hypothesis~\ref{hyp:7}, there exists a Lipschitz flow associated to $\phi$
even when~$\phi$ is not a stable almost flow. Here, we consider only Itô integrals as with the Stratonovich,
similar results require more regularity. 

Our main result below is closely connected to Proposition 4.3 in \cite{davie05a}.
We denote by $\cCloc^{1+\beta}$ the space of locally $(1+\beta)$-Hölder continuous functions. 

\begin{theorem} 
    \label{thm:diffeo}
    Assume Hypothesis~\ref{hyp:7}, and let $X$ be the unique 
    solution to the SDE~\eqref{eq:ito:1}. 
    Set $\psi_{t,s}(a)\eqdef X_t\circ X_{s}^{-1}(a)$ for any $a\in\uV=\RR^m$
    and any $(s,t)\in\rTT^2$. Then $\psi$ is almost surely a flow of $\cCloc^{1+\beta}$-diffeomorphisms,
    $0< \beta<\gamma$, 
    in the same galaxy as the almost flow $\phi[\sigma,\cdot]$ defined by \eqref{eq:ebm:davie}, 
    and $X(a)$ is the unique D-solution in
    $\cP[\phi[\sigma,\cdot],a]$. 
\end{theorem}

As the flow associated to the RDE is Lipschitz, some convergence results in \cite{brault2}
provides us a rate of convergence of discrete approximations, which is weaker as the one shown 
in Section~\ref{sec:D-sol} when stable almost flows are used. 
Here, the discrete approximation constructed from the Davie almost flow is the now classical 
\emph{Milstein scheme}~\cite{kloeden,kloeden3}. 

The pathwise rate of convergence of Itô-Taylor approximations, including the Milstein schemes, 
have been studied in \cite{talay,kloeden2,kloeden3,jentzen2}. For $\sigma\in\cC^3$, the almost sure rate of
convergence is $1-\epsilon$ for any $\epsilon>0$. Here, we consider $\sigma\in\cC^{1+\gamma}$ with 
$\gamma\leq 1$. When $\sigma\in\cCb^{2+\gamma}$, the Davie approximation is a stable almost
flow and we obtain a rate of convergence of $(2+\gamma)/p-1$ for any $p>2$, hence of order $\gamma/2-\epsilon$. 
For $\sigma\in\cCb^3$, we obtain a rate of convergence not as good as the one of P.~Kloeden and A.~Neuenkirch~\cite{kloeden2}.
Yet the main point of this section is to study the rate of convergence for an almost flow not necessarily 
stable, under weak regularity conditions.

\begin{corollary}
  \label{cor:milstein}
  Assume Hypothesis~\ref{hyp:7}, then the numerical  scheme $X^\pi$ associated
  to the Davie approximation $\phi$ given by     \eqref{eq:ebm:davie}
  with the initial condition $a\in\RR^m$
  converges almost surely to $X$, the unique solution to the SDE
  \eqref{eq:ito:1} with the rate of convergence
  $\Theta(\pi)\eqdef (\mesh\pi)^{\frac{\gamma}{2}-\epsilon},$ for all $\epsilon>0$
  such that $\epsilon<\frac{\gamma}{2}$.
\end{corollary}

\begin{proof}[Proof of Corollary~\ref{cor:milstein}]
  From Theorem~\ref{thm:diffeo}, there exists a flow $\psi$ of regularity $\cCloc^{1+\beta}$
  (\text{$0<\beta<\gamma$}) in the galaxy of $\phi$ with
  $\varpi(x)=x^{(2+\gamma)\left(\frac{1}{2}-\epsilon'\right)}$ for all $\frac{1}{2}>\epsilon'>0$. 
  Thus, according to \cite[Theorem 4.3]{brault1},
  $X^\pi$ converges almost surely to $X$
  with a rate of convergence
  $\Theta(\pi)=\pi^{\frac{\gamma}{2}-\epsilon}$ for all
  $\frac{\gamma}{2}>\epsilon>0$ and any initial condition $a\in\RR^m$.
\end{proof}
We will give two proofs of Theorem~\ref{thm:diffeo}, one being based on a regularization argument
and the second one based on the Kolmogorov-Chentsov continuity theorem.

\begin{notation}
    Let $\Omega_N$ be the ball of radius $N>0$ of $\RR^m$ and centered on $0$.
  We denote $\cG_{T,N}\eqdef \cC^0([0,T]\times\Omega_N,\RR^m)$ equipped with the
  norm
  \begin{equation}
      \label{def:norm:cg}
      \normf{\cG_{T,N}}{x}\eqdef \sup_{t\in [0,T]}\sup_{a\in\Omega_N}\abs{x_t(a)},\quad \forall x\in \cG_{T,N}.
\end{equation}
\end{notation}

\begin{theorem}[{\cite[Theorem 3.1 p.218]{kunita_saint_flour}}]
  \label{the:kunita_saint_flour}
If $\sigma$ is of class $\cC^{k+\gamma}_b$ with $\gamma\in (0,1)$
and $k\geq 1$ then the solution map $(t,a)\mapsto X_{t}(a)$ is
continuous a.s. and for all $t\in [0,T]$
$X_t(\cdot)$ is a $\cC^{k+\beta}$-diffeomorphism a.s. with $0\leq \beta<\gamma$. Moreover, for all $t\geq
0$, $a\in \RR^m$,
\begin{align}
\label{eq:DX-kunita}
\dD X_t(a)=\Id+\int_0^t\dD\sigma(X_s(a))\dD X_s(a)\vd {B_s}.
\end{align}
\end{theorem}

\begin{proof}[First proof of Theorem~\ref{thm:diffeo}]
Let $\sigma_n\in\cCb^{1+\gamma}(\RR^m)$  with $\gamma>0$
such that $\normf{\cC^{1+\gamma}_b}{\sigma_n-\sigma}\to 0$
as $n\to\infty$.
$\normhold{\gamma}{\sigma_n}\leq \mu\eqdef  \normhold{\gamma}{\sigma}$
and $\sigma_n\in\cCb^3$.

Denote by $X^n$ the solution map to $X^n_t(a)=a+\int_0^t \sigma^n(X^n_s(a))\vd B_s$.

Since $\sigma_n\in\cCb^3$, $X^n(a)$ is also a solution to the RDE  $X_t^n(a)=a+\int_0^t\sigma_n(X_s^n(a))\vd {\bB_s}$
with $\varpi(x)\eqdef x^{(2+\gamma)/p}$.
(See among others \cite{coutin-lejay3,lejay_victoir} for the Itô case and \cite{ledoux,bass,friz} for 
the Stratonovich case to which a Itô-Stratonovich correction term may be applied).
As solutions to RDE are also D-solutions, $X^n(a)$ is associated to
$\phi^n_{t,s}(a)\eqdef a+\sigma_n(a)B_{s,t} +\dD\sigma_n(a)\cdot\sigma_n(a)\bB^{(2)}_{s,t}$.

We know from \cite[Theorems 2.3 and 2.5]{kunita86a} that $\Set{X^n}_n$ converges in probability 
to~$X$ with respect to the topology generated by $\normf{\cG_{T,N}}{\cdot}$ for any $N>0$.

Besides, set $M^n_t(a)\eqdef X^n_t(a)-a=\int_0^t \sigma_n(X^n_s(a))\vd B_s$.
Recall that $\Omega_N\eqdef \Set{\abs{a}\leq N}$.
A direct application of the Burkholder-Davis-Gundy inequality on $M^n_t(a)$ shows that 
for any $p\geq2$, there exists a constant $C$ depending only on $\mu$, $p$ and $T$ such that 
\begin{equation*}
    \EE\Paren*{\sup_{t\in[0,T]} \normf{\uL^p(\Omega_N)}{M^n_t(\cdot)}^p}\leq C,\ \forall n. 
\end{equation*}
Similarly, with the Grownall lemma and the Burkholder-Davis-Gundy, one gets that 
for a constant $C'$ depending only on $\mu$, $p$ and $T$ such that 
\begin{equation*}
    \EE\Paren*{\sup_{t\in[0,T]} \normf{\uL^p(\Omega_N)}{\dD M^n_t(\cdot)}^p}\leq C',\ \forall n. 
\end{equation*}

With the Sobolev embedding theorem \cite[Theorem IX.16]{brezis}, for any integer $N$, when $p>m$ ($m$ being the dimension of the space), 
there exists a constant $K$ depending only on $N$ and $p$ such that 
\begin{equation*}
    \normf{\abs{a}\leq N}{M^n_t(a)}\leq K\Paren*{\normf{\uL^p(\Omega_N)}{M^n_t(\cdot)}
    +\normf{\uL^p(\Omega_N)}{D M^n_t(\cdot)}}.
\end{equation*}
Hence, for any $p>m$ and any $N>0$, 
\begin{equation*}
    \sup_{n\in\NN} \EE\Paren*{\normf{\cG_{T,N}}{M^n}^p}<+\infty.
\end{equation*}
This proves that $M^n$ is uniformly integrable. Therefore, $\Set{X^n}_{n\geq0}$ converges 
also in~$\uL^q$ to $X$ with respect to $\normf{\cG_{T,N}}{\cdot}$. 
Therefore,  there exists a subsequence $\Set{n_k}_{k}$ such that $\Set{X^{n_k}}_{k\geq 0}$
converges almost surely to $X$ along a subsequence with respect to~$\normf{\cG_{T,N}}{\cdot}$.

Thanks to \eqref{eq:22}-\eqref{eq:23}, each $\phi^n$ belong to $\cA[\delta,M]$ for 
a random function $\delta$ and a random constant $M$ which depend only 
on $\normp{\bB}$ and $\normhold{\gamma}{\sigma}$.
With Lemma~\ref{lem:2}, $X^{n_k}(a)\in\cP[\phi,a,L]$ for a random constant $L$ which 
is uniform in $k\geq0$ and in $a$.

Lemma~\ref{lem:convergence2} implies that $X(a)\in\cP[a,\phi,L]$.
Therefore, $X(a)$ is a D-solution associated to $\phi$. 

Since $X$ is a flow of $\cC^{1+\beta}$-diffeomorphisms for any $0\leq \beta<\gamma$, 
we set $\psi_{t,s}(a)\eqdef X_t\circ X_s^{-1}(a)$ which 
defines a flow of $\cC^{1+\beta}$-diffeomorphisms.

Since $X\in\cP[\phi,a,L]$ where $L$ does not depends on $a$, 
\begin{equation*}
    \sup_{a\in\RR^N}\abs{X_t(a)-\phi_{t,s}[\sigma,\cdot](a)}\leq L\varpi(\omega_{s,t}), \ \forall (s,t)\in\rTT^2.
\end{equation*}
Therefore, 
\begin{equation*}
    \sup_{a\in\RR^N}\abs{\psi_{t,s}(a)-\phi_{t,s}[\sigma,\cdot](X_s(a))}\leq L\varpi(\omega_{s,t}), \ \forall (s,t)\in\rTT^2.
\end{equation*}
This proves that $\phi$ and $\psi$ belong to the same galaxy.

Thanks to \eqref{eq:DX-kunita}, we see that 
$a\mapsto\psi_{t,s}(a)-a$ is locally Lipschitz for each $(s,t)$ with a uniform control
which decreases to $0$ as $T$ decreases to $0$. Hence 
$\psi$ is locally a flow of class $\cO$ (see Example~\ref{ex:1}).
Proposition~\ref{prop:uniqueness} in Appendix
shows that $X(\alpha)$ is the unique D-solution associated to $\phi[\sigma,\alpha]$
for almost all $\alpha\in A$.
\end{proof}

In the following we propose another proof of Theorem~\ref{thm:diffeo}
which is essentially based
 on the classical proof of the Kolmogorov-Chentsov criterion
 \cite[Theorem 1.8]{revuz}
 and its adaptation for the rough paths \cite[Theorem 3.1]{friz14a}.

 Let us denote, for any $a\in\RR^m$ 
    and any  $(s,t)\in\rTT^2$
  \begin{equation}
    \label{eq:Phi}
     \Psi_{s,t}(a)
     \eqdef
     X_{s,t}(a)-\sigma(X_s(a))B_{s,t}-\dD\sigma(X_s(a))\cdot \sigma(X_s(a))\bB^{(2)}_{s,t},
 \end{equation}
where $X(a)$ is the Itô solution defined by \eqref{eq:ito:1}.

\begin{lemma}[{\cite[Lemma 4.1]{davie05a}}]
\label{lem:phi-moments}
  If $\sigma\in \cC_b^{1+\gamma}$ with $\gamma\in (0,1)$,
  then for any $k>0$,
  \begin{align*}
    \EE\left(\abs{\Psi_{s,t}(a)}^k\right)\leq
    C|t-s|^{k\frac{(2+\gamma)}{2}},\quad \forall a\in\RR^m, \forall (s,t)\in\rTT^2,
  \end{align*}
  where $C$ is constant that depends only on $k$,
  $\normf{\cC_b^{1+\gamma}}{\sigma}$ and $T$ and $\Psi$ is defined \mbox{by \eqref{eq:Phi}}.
\end{lemma}

\begin{proposition}
\label{prop:ito-davie-sol}
  We assume $\sigma\in\cC^{1+\gamma}_b$ with $\gamma\in
(0,1)$. Let $k$ be the smallest integer such that
$k>\frac{6}{\gamma}$. Then, there exists a positive random
  constant $K\in L^k(A)$ 
   such that for all $(s,t)\in\rTT^2$ and
  all $a\in\RR^m$,
  \begin{align}
    \abs{\Psi_{s,t}(a)}\leq K|t-s|^{\theta},
  \end{align}
  with $\theta\eqdef 1-\frac{3}{k}+\frac{\gamma}{2}>1$. It follows that the Itô solution $X(a)$ defined
  by \eqref{eq:ito:1} is a D-solution associated to the Davie almost flow
  defined by \eqref{eq:ebm:davie}.
\end{proposition}
\begin{proof}
  We fix the integer $k$ and the real $\theta$ as in the statement of the theorem. It is well known 
  that a constant $C_k$ depending only on $k$ exists such
  that     $\EE(\abs{B_{s,t}}^k)\leq C_k |t-s|^{k/2}$ and
    $\EE(\abs{\bB^{(2)}_{s,t}}^k)\leq C_k |t-s|^k$
for any $(s,t)\in\rTT^2$.

  For an integer $n\geq0$, we set $D_n\eqdef \left\{\frac{kT}{2^n},
    k=0,\dots,2^n\right\}$ the dyadic partition of $[0,T]$. We define
  \begin{align*}
    K_n\eqdef \sup_{t\in D_n}\abs{\Psi_{t,t+T2^{-n}}(a)},\ 
    L_n\eqdef \sup_{t\in D_n}\abs{B_{t,t+T2^{-n}}}
    \text{ and }M_n\eqdef \sup_{t\in D_n}\abs{\bB^{(2)}_{t,t+T2^{-n}}}.
  \end{align*}
  It follows from Lemma~\ref{lem:phi-moments} that for any $k\geq 0$,
  \begin{align}
    \label{eq:Knk}
    \EE(K_n^k)\leq \sum_{t\in
                D_n}\EE(\abs{\Psi_{t,t+T2^{-n}}(a)}^k)
              \leq C2^n2^{-nk\frac{(2+\gamma)}{2}}.
  \end{align}
  In a same way,
  \begin{align}
    \label{eq:LM}
  \EE(L_n^k)\leq C_k2^{-n(k/2-1)}\quad \text{and}\quad \EE(M_n^k)\leq C_k2^{-n(k-1)}.    
  \end{align}

  For $s<t$ in $\bigcup_{n\in\NN} D_n$, let $m$ be an integer such that
  $2^{-(m+1)}<\abs{t-s}\leq 2^{-m}$.
  There is an integer $N$ and a partition $s=\tau_0<\tau_1<\dots<\tau_{N-1}<\tau_N=t$ of
  $[s,t]$ with the following properties

  \begin{itemize}
  \item   for each $i=0,\dots,N-1$, there $n\geq m+1$, such that $\tau_i,\tau_{i+1}$ are
  two consecutive points in $D_n$,
\item at most two consecutive points that have the same length.
\end{itemize}
Setting
$\Psi_{u,v,w}(a)\eqdef \Psi_{u,w}(a)-(\Psi_{u,v}(a)+\Psi_{v,w}(a))$
for $0\leq u\leq v\leq w\leq T$, we have
\begin{align}
 \label{eq:phitwoterm} \Psi_{s,t}(a)=\sum_{i=0}^{N-1}\Psi_{\tau_i,\tau_{i+1}}(a)+\sum_{i=0}^{N-1}\Psi_{\tau_i,\tau_{i+1},t}(a).
\end{align}
We start by bounding the first right hand side of the above equation
\begin{align}
  \label{eq:K_theta}
\frac{\left|\sum_{i=0}^{N-1}\Psi_{\tau_i,\tau_{i+1}}(a)\right|}{|t-s|^\theta}\leq
\frac{\sum_{i=0}^{N-1}\abs{\Psi_{\tau_i,\tau_{i+1}}(a)}}{2^{(m+1)\theta}}
  \leq 2\sum_{n\geq m+1}K_n2^{n\theta}\leq K_\theta,
\end{align}
where $K_\theta\eqdef  2\sum_{n\geq 0}K_n2^{n\theta}$ is a random constant in
$L^k(A)$.
Indeed,
\begin{align*}
  \normf{L^k}{K_\theta}\leq 2\sum_{n\geq 0}\normf{L^k}{K_n}2^{- n\theta}\leq
  2C\sum_{n\geq 0}2^{-n(\frac{2+\gamma}{2}-\frac{1}{k}-\theta)}.
\end{align*}
The above series is convergent because $\frac{2+\gamma}{2}-\frac{1}{k}-\theta=\frac{2}{k}>0$.

To bound the second right hand side, we note that
\begin{align}
  \label{eq:Phi-computed}
  \Psi_{u,v,w}(a)=S^{(1)}_{u,v} B_{v,w}+S^{(2)}_{u,v}B_{v,w}+S^{(3)}_{u,v}\bB^{(2)}_{v,w},
\end{align}
with
\begin{align*}
  S^{(1)}_{u,v}&\eqdef \sigma(X_v)-\sigma(X_u)-\dD\sigma(X_u)(X_v-X_u),\\
  S^{(2)}_{u,v}&\eqdef \dD\sigma(X_u)\int_u^v(\sigma(X_z)-\sigma(X_u))\vd B_z,\\
  S^{(3)}_{u,v}&\eqdef \dD\sigma(X_v)\sigma(X_v)-\dD\sigma(X_u)\sigma(X_u).
\end{align*}
We bound the moments of this three terms. For any $k>0$,
\begin{align}
  \label{eq:s1}
    \EE(\abs{S^{(1)}_{u,v}}^k)\leq
\normf{\gamma}{\dD\sigma}^k\EE(\abs{X_v-X_u}^{k(1+\gamma)})
  \leq  \normf{\gamma}{\dD\sigma}^kC_1\normsup{\sigma}^{k(1+\gamma)}\abs{v-u}^{k\frac{1+\gamma}{2}},
\end{align}
where $C_1\geq 0$ is a constant that depends only on $k$ and $\gamma$.
Similarly, 
\begin{align}
  \label{eq:s2}
  \EE(\abs{S^{(2)}_{u,v}}^k)&\leq \normsup{\dD\sigma}^{2k}C_2^2\normsup{\sigma}^k|v-u|^{k},
    \\
  \label{eq:s3}
\text{and }
\EE(\abs{S^{(3)}_{u,v}}^k)
  &\leq  \left[\normsup{\dD\sigma}^2\normsup{\sigma}^k+\normf{\gamma}{\dD \sigma}^k\normsup{\sigma}^{k(1+\gamma)}T^{k\frac{(1-\gamma)}{2}}\right]C_3\abs{v-u}^{k\frac{\gamma}{2}},
\end{align}
where $C_2$, $C_3\geq 0$ are constants depending only on $k$ and $\gamma$.
It follows from  \eqref{eq:Phi-computed}
\begin{align}
  \label{eq:Phi2}
\left|\sum_{i=0}^{N-1}\Psi_{\tau_i,\tau_{i+1},t}(a)\right|&\leq\sup_{i}\abs{B_{\tau_{i+1},t}}\sum_{i=0}^{N-1}\left(\abs{S^{(1)}_{\tau_i,\tau_{i+1}}}+\abs{S^{(2)}_{\tau_i,\tau_{i+1}}}\right)+
\sup_i\abs{\bB^{(2)}_{\tau_i,t}}\sum_{i=0}^{N-1}\abs{S^{(3)}_{\tau_i,\tau_{i+1}}}.
\end{align}
Yet, we have
\begin{align}
  \label{eq:supB}
  \sup_i\abs{B_{\tau_i,t}}\leq \sum_{i=0}^{N-1}\abs{B_{\tau_i,\tau_{i+1}}}\leq
  2\sum_{n\geq m+1}L_n.
\end{align}
Using Chen's relation
\begin{align}
  \label{eq:supBB}
  \sup_i\abs{\bB^{(2)}_{\tau_i,t}}&\leq
    \sum_{i=0}^{N-1}\abs{\bB^{(2)}_{\tau_i,\tau_{i+1}}}+\sup_{i}\abs{B_{\tau_{i+1},t}}\sum_{i=0}^{N-1}\abs{B_{\tau_i,\tau_{i+1}
    }}
  &\leq 2\sum_{n\geq m+1}M_n+\left(2\sum_{n\geq m+1}L_n\right)^2.
\end{align}
Thus, combining \eqref{eq:Phi2}, \eqref{eq:supB} and \eqref{eq:supBB},
\begin{multline*}
  \left|\sum_{i=0}^{N-1}\Psi_{\tau_i,\tau_{i+1},t}(a)\right|
  \leq 4\sum_{n\geq m+1}L_n\sum_{n\geq
    m+1}\left(S^{(1)}_n+S^{(2)}_n\right)\\
    +\left(2\sum_{n\geq m+1}M_n
    +\left(2\sum_{n\geq m+1}L_n\right)^2\right)\left(\sum_{n\geq m+1}S^{(3)}_n\right),
\end{multline*}
where $S^{(\ell)}_n\eqdef \sup_{t\in D_n}\abs{S^{(\ell)}_{t,t+T2^{-n}}}$ for
$l\in\{1,2,3\}$.
We show with \eqref{eq:s1}, \eqref{eq:s2} and \eqref{eq:s3}, in a same
way as for $K_n$, that for $\ell\in\Set{1,2}$,
\begin{equation}
  \label{eq:sl}
\EE\left(\left[S^{(\ell)}_n\right]^k\right)\leq
  C_42^{-n\left(\frac{k(1+\gamma)}{2}-1\right)},
  \text{~and~} \EE\left(\left[S^{(3)}_n\right]^k\right) \leq C_52^{-n\left(\frac{k\gamma}{2}-1\right)},
\end{equation}
where $C_4$, $C_5$ are constants that depend on
$\normf{\cC^{1+\gamma}_b}{\sigma}$, $k$, $\gamma$ and $T$.

We recall that $\theta\eqdef 1-\frac{3}{k}+\frac{\gamma}{2}$, then $\theta>1$ and there exists a constant $\theta_1\in \left(\frac{1}{2}-\frac{1}{k},\frac{1}{2}\left(\theta+\frac{1}{k}-\frac{\gamma}{2}\right)\right)$. We have
\begin{align}
  \label{eq:K'_theta}
\frac{\left|\sum_{i=0}^{N-1}\Psi_{\tau_i,\tau_{i+1},t}(a)\right|}{\abs{t-s}^{\theta}}\leq  K'_\theta,
\end{align}
where 
\begin{multline}
  \label{eq:K'series}
  K'_\theta\eqdef 4\sum_{n\geq 0}L_n2^{n\theta_1}
      \sum_{n\geq 0}\left(S^{(1)}_n+S^{(2)}_n\right)2^{n(\theta_2-\theta_1)}
  \\
    +\left(2\sum_{n\geq 0}M_n2^{n2\theta_1}+\left(2\sum_{n\geq 0}
    L_n2^{n\theta_1}\right)^2\right)\left(\sum_{n\geq 0}S^{(3)}_n2^{n(\theta-\theta_1)}\right).
\end{multline}
The constant $K'_\theta$ is random variable in $\uL^k(A)$. Indeed,
using \eqref{eq:LM}, \eqref{eq:sl} and our choice of
$k$, $\theta_1$, $\theta$, we check that
the right hand side of \eqref{eq:K'series} contains only
converging series in $\uL^k(A)$.

Setting $K\eqdef K_\theta+K'_\theta$ and using \eqref{eq:phitwoterm}, \eqref{eq:K_theta},
\eqref{eq:K'_theta}, we obtain that for all $s<t$ in $\bigcup_{n\in\NN} D_n$,
$\abs{\Psi_{s,t}(a)}\leq K\abs{t-s}^\theta$,
with $K\in \uL^k(A)$.
By continuity of $(s,t)\mapsto \Psi_{s,t}(a)$ the above
estimation is true for all $(s,t)\in\rTT^2$. This concludes the proof.
\end{proof}

\begin{proof}[Second proof of Theorem~\ref{thm:diffeo}]

According to Proposition~\ref{prop:ito-davie-sol}, when
$\sigma\in\cC_b^{1+\gamma}$ with $\gamma\in (0,1)$, the Itô solution
$X(a)$ is a D-solution associated to Davie almost flow $\phi$.
More precisely, $X\in\cP[\phi,a,K]$, with a random constant $K$ that
does not depend on $a$. Then we conclude as in the first proof of Theorem~\ref{thm:diffeo}.
\end{proof}

\appendix

\section{Boundedness of solutions}

\subsection{Almost flows with linear or almost growth}

\label{sec:aflinear}

In \cite{brault1}, almost flows are not necessarily bounded.
In this appendix, we consider almost flows for which \eqref{eq:def:3} and \eqref{eq:def:2} are replaced by
\begin{align}
     \label{def:unbd:1}
     \abs{\phi_{t,s}(a)-a}&=\abs{\widehat{\phi}_{t,s}}\leq \delta_{t-s}N(a),\\
     \label{def:unbd:2}
     \abs{\diff\phi_{t,s,r}(a)}&\leq N(a)\varpi(\omega_{s,t})
 \end{align}
for a $\gamma$-Hölder function $N:\uV\to\RR_+$ such that $\inf_{a} N(a)>0$. 

For any $\pi$ of $\TT$ and any $(s,t)\in\rTT^2$, we write
$\phi^\pi_{t,s}\eqdef \phi_{t,t_j}\circ\dots\circ\phi_{t_i,s},$
where $[t_i,t_j]$ is the biggest interval of such kind contained in $[s,t]$.

\begin{theorem}[{\cite[Theorem 1]{brault1}}] 
    \label{thm:control}
    There exists a time horizon $T$ and constant $L\geq 1$ depending only on $\normlip{N}$, 
    $\delta$, $\omega$ and $\varpi$ such that 
    \begin{equation*}
	\abs{\phi^\pi_{t,s}(a)-\phi_{t,s}(a)}\leq L N(a) \varpi(\omega_{s,t}), \forall (s,t)\in \rTT^2
    \end{equation*}
    uniformly in the partition $\pi$ of $\TT$.
\end{theorem}

Now, let us fix $T$ as in Theorem~\ref{thm:control}, 
$R\geq 0$ and set $\Omega(R)\eqdef \Set{a\in\uV\given \abs{a}\leq R}$. 

Let us consider a path $y\in\cC(\TT,\uV)$ such that 
$\abs{y_{t,s}-\phi_{t,s}(y_s)}\leq K\varpi(\omega_{s,t})$
for any $(s,t)\in\rTT^2$ and $y_0=a$ for some $a\in\Omega(R)$. 
With \eqref{def:unbd:1}-\eqref{def:unbd:2}, 
\begin{equation*}
    \abs{y_t(a)}\leq \abs{a}+\delta_TN(a)+K\varpi(\omega_{0,t}).
\end{equation*}
With $\overline{N}_R\eqdef \sup_{\abs{a}\leq R} N(a)$, 
\begin{equation}
    \label{eq:bnd:y}
    \normsup{y}\leq R+\delta_T\overline{N}_R+K\varpi(\omega_{0,T}). 
\end{equation}

Combining the above inequality with Theorem~\ref{thm:control} leads to 
the following uniform control.

\begin{corollary}
    \label{cor:1}
    If the sequence of paths $\Set{t\mapsto \phi^\pi_{t,0}(a)}_{\pi}$ converges
    to a path $y\in\cC(\TT,\uV)$, which is a D-solution, then $\normsup{y(a)}\leq R'$
    with $R'\eqdef R+\delta_T\overline{N}_R+K\overline{N}_R\varpi(\omega_{0,T})$. 
\end{corollary}

Combining Corollary~\ref{cor:1} with \eqref{eq:bnd:y} and Proposition 10 in \cite{brault2},
we obtain a truncation argument. Thus, as we consider starting points in a bounded
set, we assume that stable almost flows are bounded without loss of generality.

\begin{corollary}[Truncation argument]
    \label{cor:2}
    Let $\phi$ be an almost flow satisfying \eqref{def:unbd:1}-\eqref{def:unbd:2}
    and $\psi$ be a stable almost flow.  We assume that $\phi=\psi$ on $\Omega(R')$
    with $R'$ defined above. Let us consider a D-solution $y$ for $\phi$ with $y_0=a$, $a\in\Omega(R)$
    with a constant $K=L\overline{N}_R$. Then $y$ 
    is a D-solution for $\psi$ and is unique.
\end{corollary}

\subsection{Boundedness of solutions}

We now give some general results about uniform boundedness of the solutions. 
For this, we add a hypothesis on the structure of the almost flows.

\begin{hypothesis}
    \label{hyp:6}
Let $\Lambda:\RR_+\to\RR_+$ be a continuous, non-decreasing 
function such that 
\begin{equation*}
    \Lambda(0)=0\text{ and }\lim_{x\to 0}\Theta(x)=0
    \text{ with }
    \Theta(x)\eqdef\frac{\varpi(x)}{\Lambda(x)}.
\end{equation*}
\end{hypothesis}

Typically, we use $\Lambda(x)=x^{1/p}$ for some $p\geq 1$. This hypothesis
is satisfies when considering YDE and RDE.

For $y\in\cC([0,T],\uV)$, we define
\begin{equation*}
    \normL{y}\eqdef \sup_{(s,t)\in\rTT^2}\frac{\abs{y_t-y_s}}{\Lambda(\omega_{s,t})}.
\end{equation*}

\begin{notation}
    Let  $\cF_\Lambda(\delta)$ be the elements of $\cF[\delta]$ 
satisfying for some constants $C_\Lambda$ and $R_0$, 
$\phi_{t,s}=\id+\widehat{\phi}_{t,s}$, 
\begin{gather}
    \label{eq:cont:2}
    \Phi(R)\eqdef \sup_{\abs{a}\leq R}\normL{\widehat{\phi}_{t,s}(a)}\leq C_\Lambda R,\ \forall R\geq R_0.
\end{gather}
\end{notation}


\begin{proposition} 
    Let $\phi\in\cF_\Lambda[\delta]$. Let $y\in\cP[\phi,a,K]$. 
    When $C_\Lambda\Lambda(\omega_{0,T})\leq 1/2$, 
    \begin{equation}
	\label{eq:cont:3}
	\normsup{y}\leq 2\abs{a}+2K\varpi(\omega_{0,T})+2C_\Lambda R_0 \Lambda_{0,T}.
    \end{equation}
    In addition, 
    \begin{equation}
	\label{eq:cont:4}
	\normL{y}\leq C_\Lambda\max\Set{\normsup{y},R_0}+K\Theta(\omega_{0,T}).
    \end{equation}
\end{proposition}
\begin{proof} With \eqref{eq:1} and \eqref{eq:cont:2}, 
    \begin{equation*}
	\abs{y_t-y_s}\leq K\varpi(\omega_{s,t})+C_\Lambda \max\Set{R_0,\normsup{y}}
	\Lambda(\omega_{s,t}).
    \end{equation*}
    In particular, 
    \begin{equation*}
	\normsup{y}\leq \abs{a}+K\varpi(\omega_{0,T})+C_\Lambda(R_0+\normsup{y})\Lambda(\omega_{0,T}). 
    \end{equation*}
    For $C_\Lambda\Lambda(\omega_{0,T})\leq 1/2$, this leads
    to \eqref{eq:cont:3}. Since
    \begin{equation*}
	\abs{y_t-y_s-\widehat{\phi}_{t,s}(y_s)}\leq K\varpi(\omega_{s,t}), 
    \end{equation*}
    we obtain \eqref{eq:cont:4}.
\end{proof}

\subsection{Uniqueness of D-solutions associated to flows of class $\cO$}

\label{sec:uniqueness}

In our setting, we have not assumed that a flow is continuous. 
If a flow is locally of class $\cO$, then the associated D-solution is unique. 
We adapt the proof of \cite[Proposition 4.3]{davie05a} in our setting.

\begin{proposition}[Uniqueness of D-solutions associated to flows of class $\cO$]
    \label{prop:uniqueness}
    Let $\phi$ be a flow locally of class $\cO$ and $y$ be a D-solution 
    in $\cP[\phi,a,K]$. Then $y_t=\phi_{t,0}(a)$ for any $t\in\TT$
    and is then unique.
\end{proposition}

\begin{proof}
    As $y$ lives in a bounded set, we assume without loss of generality 
    that $\phi$ is globally of class $\cO$ as we use only local controls
    on the modulus of continuity of $\phi_{t,s}$.

    Let $\pi=\Set{t_i}_{i=0}^n$
    be a partition of $[0,t]$, $t\leq T$. Let us set $y_k\eqdef y_{t_k}$ and 
    $v_k=\phi_{t,t_k}(y_k)$. This way, $v_n=y_t$ while $v_0=\phi_{t,0}(a)$.
    Using a telescoping series, 
    \begin{equation*}
	v_{n}-v_0
	=\sum_{k=0}^{n-1}(v_{k+1}-v_k)
	=\sum_{k=0}^{n-1}\phi_{t,t_{k+1}}(y_{k+1})
	-\phi_{t,t_{k+1}}\circ\phi_{t_{k+1},t_k}(y_k).
    \end{equation*}
    Set $d_k\eqdef \abs{y_{k+1}-\phi_{t_{k+1},t_k}(y_k)}$. As $y\in \cP[\phi,a,K]$, $d_k\leq K\varpi(\omega_{k,k+1})$.
    As $\phi_{t,t_k}$ is of class $\cO$, 
    \begin{equation*}
	\abs{v_n-v_0}
	\leq \sum_{k=0}^{n-1}(d_k+\delta_T\normO{\widehat{\phi}}(1+K)\varpi(\omega_{t_k,t_{k+1}}))
    \end{equation*}
    so that 
    \begin{equation*}
	\abs{y_t-\phi_{t,s}(a)}=\abs{v_n-v_0} 
	\leq \Paren*{K+\normO{\widehat{\phi}}(1+K)\delta_T}\sum_{k=0}^{n-1}\varpi(\omega_{k,k+1})
	\xrightarrow[\mesh{\pi}\to 0]{}0 
    \end{equation*}
    since $\varpi(x)/x$ converges to $0$ as $x$ decreases to $0$.
\end{proof}

\bigskip
\noindent
\textbf{Acknowledgement.} The authors wish to thank Laure Coutin for her careful
reading and interesting discussions regarding the content of this article.
The first author thanks the Center for Mathematical Modeling, Conicyt fund AFB 170001.



\end{document}